\documentclass[11pt,reqno]{amsart}
\usepackage[body={7in,9in},centering]{geometry}
\usepackage{amsmath,amssymb,amsthm,mathrsfs}
\usepackage{color,xcolor}
\definecolor{cobalt}{RGB}{61,99,181}
\usepackage[colorlinks,citecolor=cobalt,linkcolor=cobalt]{hyperref}
\usepackage{exscale}
\usepackage{relsize}
\newtheorem{thm}{Theorem}[section]
\newtheorem{defi}[thm]{Definition}
\newtheorem{cor}[thm]{Corollary}
\newtheorem{lem}[thm]{Lemma}

\newtheorem{prop}[thm]{Proposition}

\numberwithin{equation}{section}

\date{\today}

\makeatletter

\newcommand{\Rmnum}[1]{\expandafter\@slowromancap\romannumeral #1@}
\makeatother

\begin{document}

\title[Toeplitz algebras over Fock and Bergman spaces]{Toeplitz algebras over Fock and Bergman spaces}

\author[Shengkun Wu]{Shengkun Wu\textsuperscript{1}}
\address{\textsuperscript{1} College of Mathematics and Statistics, Chongqing University, Chongqing, 401331, P. R. China}
\email{shengkunwu@foxmail.com}
\author[Xianfeng Zhao]{Xianfeng Zhao\textsuperscript{2}}
\address{\textsuperscript{2} College of Mathematics and Statistics, Chongqing University, Chongqing, 401331, P. R. China}
\email{xianfengzhao@cqu.edu.cn}



\begin{abstract}
In this paper, we investigate Toeplitz subalgebras generated by certain class of Toeplitz operators on the $p$-Fock space and the $p$-Bergman space with $1<p<\infty$. Let BUC($\mathbb C^n$) and BUC($\mathbb B_n$) denote the collections of bounded
uniformly continuous functions on $\mathbb C^n$ and $\mathbb B_n$ (the unit ball in $\mathbb C^n$), respectively. On the $p$-Fock space, we show that the Toeplitz algebra which has a translation invariant closed subalgebra of BUC($\mathbb C^n$)  as its set of symbols is linearly generated by Toeplitz operators with the same space of symbols. This  answers an open question recently raised by Fulsche \cite{Robert}. On the $p$-Bergman space, we study Toeplitz algebras with symbols in some translation invariant closed subalgebras of BUC($\mathbb B_n)$. In particular, we obtain that the Toeplitz algebra generated by all Toeplitz operators with symbols in BUC($\mathbb B_n$) is equal to the closed linear space generated by Toeplitz operators with such symbols.  This generalizes the corresponding result for the case of $p=2$ obtained by Xia \cite{Xia2015}, which was proven by a different approach.
\end{abstract} \maketitle

\tableofcontents

\section{Introduction}\label{Intro}
We begin with some basic notations and knowledge about Fock spaces, Bergman spaces and Toeplitz operators on such spaces.
For a positive parameter $t$, let
$$d\mu_{t}(z)=\frac{1}{(\pi t)^n} e^{-\frac{|z|^2}{t}}dV(z)$$
be the Gaussian measure on $\mathbb{C}^n$, where $dV(z)$ denotes the Lebesgue  measure on $\mathbb{C}^n$.   For $1<p<\infty$, the $p$-Fock space $F_t^p$ is the space of entire functions $f$ on $\mathbb C^n$
which are $p$-integrable with respect to the Gaussian measure with a parameter $\frac{2t}{p}>0$, i.e.,
$$\|f\|_{F_t^p}= \Big[\int_{\mathbb{C}^n}|f(z)|^p d\mu_{2t/p}(z)\Big]^{\frac{1}{p}}<\infty.$$
Let $q$ be the conjugate number of $p$, i.e., $\frac{1}{p}+\frac{1}{q}=1$. Then the  dual space of $F_t^p$ is $F_t^q$ and the  duality  pairing  is given by
$$\langle f,g\rangle_t=\int_{\mathbb{C}^n} f(z)\overline{g(z)}d\mu_{t}(z).$$

For $p=2$, it is well-known that $F_t^2$ is a reproducing kernel Hilbert space with the reproducing kernel
$$K^t_z(w)=e^{\frac{w\cdot \overline{z}}{t}},  \ \ \ z, w\in \mathbb C^n,$$
where $$w\cdot \overline{z}=w_1\overline{z_1}+\cdots+w_n\overline{z_n}$$ for $w=(w_1, \cdots, w_n)$ and $z=(z_1, \cdots, z_n)$.
The normalized reproducing kernel is given by
$$k_z^t(w)=\frac{K^t_z(w)}{\|K_z^t\|_{F^2_t}}=e^{\frac{w\cdot \overline{z}}{t}-\frac{1}{2t}|z|^2},\ \ \ z, w\in \mathbb C^n.$$
Moreover, one can check easily  that
$$\|f\|_{F^p_t}=\Big[\frac{1}{(\pi t)^n}\int_{\mathbb{C}^n}|\langle f,k_z^t\rangle_t|^pdV(z)\Big]^{\frac{1}{p}}$$
for $1\leqslant p<\infty$.

With the normalized reproducing kernel for the Fock space $F_t^2$, the  Berezin transform of an operator $T$ on $F_t^p$ is defined by
$$ \widetilde{T}(z)=\langle Tk_z^t,k_z^t\rangle_t,  \ \ \ z\in \mathbb C^n.$$
For $\varphi\in L^{\infty}(\mathbb{C}^n, dV)$,  the Toeplitz operator $T_\varphi$ with symbol $\varphi$ on $F_t^p$  is defined by
$$T_\varphi f(z)=\int_{\mathbb{C}^n}\varphi(w)f(w)\overline{K^t_{z}(w)}d\mu_t(w).$$
Then $T_\varphi$ is bounded on $F_t^p$, i.e., $T_\varphi$ is bounded from $F_t^p$ to $F_t^p$.

Let $\mathbb{B}_n=\{z\in \mathbb C^n: |z|<1\}$ denote the open unit ball in $\mathbb{C}^n$. Let $dv$ denote the normalized Lebesgue measure on $\mathbb B_n$ with the normalization $v(\mathbb B_n) = 1$.  For $1<p<\infty$, the $p$-Bergman space $L_a^p$  is the collection of  all  holomorphic functions that are $p$-integrable with respect to $dv$. The norm on the Banach space $L^p_a$ is given by
$$\|f\|_{p}=\Big[\int_{\mathbb{B}_n}|f(z)|^pdv(z)\Big]^{\frac{ 1}{p}}.$$
Furthermore, the  dual space of $L_a^p$ is $L_a^q$ under the standard  duality  pairing:
$$\langle f,g\rangle=\int_{\mathbb{\mathbb B}_n} f(z)\overline{g(z)}dv(z),$$
where $\frac{1}{p}+\frac{1}{q}=1$. Recall that the reproducing kernel for the Bergman space $L_a^2$ is given by
$$K_{z}(w)=\frac{1}{(1-w\cdot \overline{z})^{n+1}},  \ \ \ z, w\in \mathbb B_n.$$
Similar to the setting of $p$-Fock space, the Toeplitz operator $T_\varphi$ with symbol $\varphi \in L^{\infty}(\mathbb B_n, dv)$ on  $L_a^p$ can  also  be  defined  via the reproducing kernel:
$$T_\varphi f(z)=\int_{\mathbb{B}_n}\varphi (w)f(w)\overline{K_{z}(w)}dv(w),$$
which is bounded on $L_a^p$.

Toeplitz algebras over Bergman spaces and Fock spaces have been investigated  by many authors for a long time, see for example \cite{Englis}, \cite{suarez2005, suarez}, \cite{Bauer}, \cite{XiaZheng}, \cite{wick2014}, \cite{Xia2015, Xia2017, Xia2018}, \cite{Robert} and \cite{Hagger}. Let us fix more notations  before going to review the background about the investigation of Toeplitz algebras over these two function spaces.

By the \emph{full Toeplitz algebra}, we mean that the Banach algebra generated by all Toeplitz operators with (essentially) bounded symbols. Let $\mathcal{J}$ be a subset of $L^\infty(\mathbb C^n, dV)$ and  $\mathcal{T}(\mathcal{J})$ be the Banach algebra generated by
$\{T_\varphi: \varphi\in \mathcal{J}\}.$ In this paper, we shall call  $\mathcal{T}(\mathcal{J})$ the Toeplitz algebra generated by all Toeplitz operators with symbols in  $\mathcal{J}.$
Moreover, we use  $\mathcal{T}_{lin}(\mathcal{J})$ to denote the closed linear space generated by  Toeplitz operators with symbols in $\mathcal{J}$. Let $\mathrm{BUC}(\mathbb{C}^n)$ denote the set of bounded  functions on $\mathbb{C}^n$ which are uniformly continuous with respect to the Euclidean metric. For $f\in F_t^p$ and $z\in \mathbb{C}^n$, the translation $\alpha_z$ is defined by
$$(\alpha_z f)(w):=f(w+z), \ \ \ \ w\in \mathbb C^n.$$
We say that $\mathcal{J}\subset \mathrm{BUC}(\mathbb{C}^n)$ is \emph{translation invariant on $\mathbb C^n$} if $\alpha_zf\in \mathcal{J}$ for all $f\in\mathcal{J}$ and $z\in\mathbb{C}^n$. Note that $\mathrm{BUC}(\mathbb{C}^n)$ is translation invariant. In addition, the Berezin transform of any bounded linear operator on $F^p_t$ is also in $\mathrm{BUC}(\mathbb{C}^n)$, see \cite[Lemma 2.8]{Robert} if necessary.

In the setting of the $p$-Bergman space,  let $\mathcal{I}$ be a subset of $L^\infty(\mathbb B_n, dv)$ and  $\mathcal{T}^b(\mathcal{I})$ be the Banach algebra generated by
$\{T_\varphi: \varphi\in \mathcal{I}\}.$ In the following, we shall call $\mathcal{T}^b(\mathcal{I})$ the Toeplitz algebra generated by all Toeplitz operators with symbols in $\mathcal{I}$. We denote  $\mathcal{T}^b_{lin}(\mathcal{I})$ by the closed linear space generated  Toeplitz operators  with symbols in $\mathcal{I}$.  We use  $\mathrm{BUC}(\mathbb{B}_n)$ to denote the set of bounded  functions on $\mathbb B_n$ which are uniformly continuous with respect to the \emph{Bergman  metric} (which will be introduced in Section \ref{Bergman-I-R}). For $f\in\mathrm{BUC}(\mathbb{B}_n)$ and $z\in \mathbb B_n$, we define $\tau_z$ on $\mathrm{BUC}(\mathbb{B}_n)$ as
\begin{align}\label{tau}
(\tau_zf)(u)=f(\varphi_{u}(z)), \ \ \ u\in \mathbb B_n,
\end{align}
where $\varphi_u$ is the M\"{o}bius transform of $\mathbb B_n$ that interchanges 0 and $u$. It is worth noting that $\tau_zf$ is not defined by $\tau_zf=f\circ \varphi_z$. However,  this ``unnatural" definition of $\tau_z$ in (\ref{tau}) plays an important role in the study of some Toeplitz subalgebras over the Bergman space of $\mathbb B_n$. For a subset $\mathcal{I}$ of  $\mathrm{BUC}(\mathbb{B}_n)$, we say that  $\mathcal{I}$ is \emph{translation invariant on $\mathbb B_n$} if $\tau_zf\in \mathcal{I}$ for all $f\in \mathcal{I}$ and $z\in \mathbb{B}_n$.

In 2007, Su\'{a}rez \cite{suarez} obtained that the full Toeplitz algebra over the $p$-Bergman space $L_a^p$ equals  the Toeplitz algebra with symbols in BUC$(\mathbb B_n)$ via the $n$-Berezin transform. Then Bauer and Isralowitz \cite{Bauer} established a similar result for the full Toeplitz algebra over weighted Fock spaces.  In \cite{XiaZheng}, Xia and Zheng defined the sufficiently localized operator on $F^2_t$ and  characterized  the compactness of operators in the $C^*$-algebra generated by sufficiently localized operators. Furthermore, they obtained that the $C^*$-algebra generated by sufficiently localized operators contains the full Toeplitz algebra over the Fock space $F^2_t$. Later, Isralowitz, Mitkovski  and  Wick \cite{wick2014} introduced the weakly localized operators on $L_a^p$ (and also $F^p_t$) and showed that such class of  operators forms an algebra and its closure also contains the full Toeplitz algebra over $L_a^p$ ($F^p_t$). Based on the study of the $C^*$-algebra generated by weakly localized operators, Xia \cite{Xia2015} showed that the full  Toeplitz algebra over $L_a^2$ is equal to the norm closure of $\big\{T_\varphi: \varphi\in L^\infty(\mathbb B_n, dv)\big\}$. Indeed, in the cases of the Bergman space $L_a^2$ and the Fock space $F_1^2$, Xia \cite{Xia2015} obtained that the full  Toeplitz algebra  coincides with the $C^*$-algebra generated by the class of weakly localized operators. In 2020, using a correspondence theory
of translation invariant symbol on $\mathbb C^n$ and operator spaces,   Fulsche \cite{Robert} showed  that the full  Toeplitz algebra over the $p$-Fock space $F_t^p$ is the norm closure of all Toeplitz operators with symbols in BUC($\mathbb C^n$), which  generalizes the result obtained by
Xia \cite{Xia2015} in the case of $p=2$.  Moreover, Fulsche \cite{Robert}  proved  that every Toeplitz algebra which
has a translation and $U$-invariant $C^*$-subalgebra of BUC$(\mathbb C^n)$ as its set of symbols is linearly
generated by Toeplitz operators with the same space of symbols. Recently, by using a characterization of Toeplitz algebras over  the $p$-Fock space $F_t^p$ \cite{Robert},  Hagger \cite{Hagger} established that the full Toeplitz algebra over $F_t^p$ coincides with
each of the algebras generated by band-dominated, sufficiently localized and weakly localized
operators, respectively.

In this paper,  we will mainly  consider  the Toeplitz algebras generated by Toeplitz operators with symbols in some translation invariant closed subalgebras of $\mathrm{BUC}(\mathbb{C}^n)$ and $\mathrm{BUC}(\mathbb{B}_n)$ on the $p$-Fock space and the $p$-Bergman space, respectively.
Recall  that the orthonormal  basis of $L^2_a$, and the integral representation in terms of certain sum of rank-one operators over a lattice for a class of Toeplitz operators  are the crucial ingredients in Xia's  approach \cite{Xia2015}. Although the techniques used in \cite{Xia2015} can not be applied directly to treat  Toeplitz operators with translation invariant symbols on the Banach spaces $L_a^p$ and $F_t^p$, some useful ideas (such as the usage of separated sets in the unit ball, weakly localized operators and integral representations for Toeplitz operators) provide  a great inspiration to us. Noting that the automorphism group of $\mathbb{C}^n$ is commutative, thus the convolution of two functions (or two operators)  has some nice properties on Fock spaces and which can be used to approximate an operator in the Toeplitz algebra  by a sequence of Toeplitz operators \cite{Robert}. However, the automorphism group of $\mathbb{B}_n$ is not commutative, so the techniques used in \cite{Robert}  may do not work for the case of Bergman spaces over the unit ball.

In order to characterize  Toeplitz algebras  with symbols in a subset of BUC$(\mathbb C^n)$ (BUC$(\mathbb B_n)$) over the $p$-Fock space ($p$-Bergman space), we will first establish an  integral representation for weakly localized operators on the $p$-Fock space ($p$-Bergman space).  Based on this integral representation, we are able to further study  weakly localized operators on $F_t^p$ and  $L_a^p$   via the Berezin transform. We now give a short outline of the rest of the paper.  In Section \ref{Fock}, we show  that the Toeplitz algebra generated by Toeplitz operators on the $p$-Fock space $F_t^p$ with symbols in a translation invariant closed subalgebra of BUC$(\mathbb C^n)$ is linearly generated by Toeplitz operators with the same space of
symbols, see Theorem \ref{focktop}. This answers an open  question posed in \cite{Robert} by using different
methods. In Section \ref{Bergman-I-R}, we obtain an  integral representation for \emph{$s$-weakly localized operators} (see Definition \ref{s-weakly}) on the $p$-Bergman space $L_a^p$, see Theorem \ref{integralrepresantation}. Then we apply this
integral representation to study Toeplitz algebras with symbols in some  translation invariant subalgebra $\mathcal I\subset \mathrm{BUC}(\mathbb{B}_n)$ in Section \ref{Bergman-T-A}. In particular, we obtain in Theorem \ref{pbergthm}  that $\mathcal{T}^b[\mathrm{BUC}(\mathbb{B}_n)]$ and $\mathcal{T}^b_{lin}{[\mathrm{BUC}(\mathbb{B}_n)]}$ are both equal to the norm closure of the collection of $s$-weakly localized operators on $L_a^p$;  $\mathcal{T}^b[C_0(\mathbb{B}_n)]$ and $\mathcal{T}^b_{lin}[C_0(\mathbb{B}_n)]$ are both equal to the ideal of compact operators on $L_a^p$, where $C_0(\mathbb{B}_n)$ is a (translation invariant) subalgebra of $\mathrm{BUC}(\mathbb B_n)$ consisting of functions $f$ with $f(z)\rightarrow 0$ as $|z|\rightarrow 1$. This  generalizes a result obtained by Xia in the case $p=2$, see \cite[Theorem 1.5]{Xia2015}.

In the following, $T^*$ denotes the Banach space adjoint of the bounded linear operator  $T$ on $F_t^p$ or $L_a^p$. In addition,  the notation $A\lesssim B$ for two nonnegative quantities $A$ and $B$ means that there is some inessential  constant $C>0$ such that $A \leqslant CB$.

\section{Toeplitz algebras over the $p$-Fock space}\label{Fock}
This section is devoted to the study of the Toeplitz algebra $\mathcal{T} (\mathcal{J})$ over $F_t^p$ with $1<p<\infty$, where $\mathcal{J}$ is  a translation invariant closed subalgebra of $\mathrm{BUC}(\mathbb{C}^n)$. The main result of this section is the following theorem, which shows  the assumption that the $U$-invariance  of $\mathcal J$ (i.e., $U\mathcal{J}\subset \mathcal{J}$, where $Uf(z)=f(-z)$ for $f\in \mathcal J$ and $z\in \mathbb C^n$) and the self-adjointness of $\mathcal J$ are in fact unnecessary. This  answers an open problem recently raised  by Fulsche in \cite[page 39]{Robert}.
\begin{thm}\label{focktop}
Let $\mathcal{J}$ be a translation invariant closed subalgebra of $\mathrm{BUC}(\mathbb{C}^n)$,  then
$$\mathcal{T}(\mathcal{J})=\mathcal{T}_{lin}(\mathcal{J})$$
holds on $F^p_t$. Moreover, if $\mathcal{J}_1$ is a translation invariant closed ideal of $\mathcal{J}$, then $\mathcal{T}(\mathcal{J}_1)$ is a two-sided ideal in $\mathcal{T}(\mathcal{J})$.
\end{thm}
The proof of the above theorem relies on the integral representation for \emph{weakly localized operators} on $F_t^p$.  Let us recall the definition of this class of  operators on the $p$-Fock space $F_t^p$, which was first introduced by Isralowitz, Mitkovski and Wick \cite[Definition 1.1]{wick2014}.
\begin{defi}\label{weaklocal}
Let $T$ be a bounded linear operator  on $F_t^p$ and $T^*$ be the Banach space adjoint of $T$. Then $T$ is said to be  weakly localized if it satisfies the following four conditions:
$$\sup_{z\in \mathbb C^n}\int_{\mathbb C^n}|\langle Tk^t_z,k^t_w\rangle_t|dV(w)<\infty, \text{\quad}\sup_{z\in \mathbb C^n}\int_{\mathbb C^n}|\langle T^*k^t_z,k^t_w\rangle_t|dV(w)<\infty,$$
$$\lim_{r\rightarrow\infty}\sup_{z\in \mathbb C^n}\int_{\mathbb C^n\setminus B(z,r)}|\langle Tk^t_z,k^t_w\rangle_t|dV(w)=0, \text{\quad}\lim_{r\rightarrow\infty}\sup_{z\in \mathbb C}\int_{\mathbb C\setminus B(z,r)}|\langle T^*k^t_z,k^t_w\rangle_t|dV(w)=0,$$
where $B(z,r)$ denotes the ball $\{w\in \mathbb C^n: |w-z|<r\}$ in $\mathbb C^n$ with center $z$ and radius $r$.  We denote the  collection of weakly localized operators on $F_t^p$  by $\mathcal{A}^{p}$.
\end{defi}

As the proof of Theorem \ref{focktop} is long, we will divide the proof into several steps and complete the details in the end of this section. The  main three steps of our approach are the following.
\begin{itemize}
\item  [(1)] First, we establish  an integral representation  for  weakly localized operators on $F_t^p$ (Theorem  \ref{fockintrep});\\
\item [(2)] Based on  the integral representation mentioned above, we  show that a weakly localized operator belongs to the space  $\mathcal{T}_{lin}(\mathcal{J})$ if its  Berezin transform  is in $\mathcal{J}$ (Proposition \ref{fockweak}); \\
\item [(3)] For a finite product of Toeplitz operators with symbols in $\mathcal{J}$,  we  show that its Berezin transform  belongs to $\mathcal{J}$ (Proposition \ref{fockproduct}).
\end{itemize}
Noting that the  finite sum of finite products of Toeplitz operators with symbols in $\mathcal{J}$  is weakly localized, and  $\mathcal{T}_{lin}(\mathcal{J})\subset\mathcal{T}(\mathcal{J})$, then Theorem  \ref{focktop} will follow immediately from (1)-(3).

In order to establish an integral representation for weakly localized operators on  $F^p_t$, first we need to recall some properties of Weyl operators on  Fock spaces. For $z\in \mathbb C^n$, the Weyl operator $W_z$ is defined by
$$ W_zf(w)=k_z^t(w)f(w-z), \ \ \  w\in \mathbb C^n,$$
where $f\in F^p_t$. It is easy to check that the following hold for all $z, w\in \mathbb C^n$: $W_z^*=W_{-z}$ and
\begin{equation}\label{weyl}
W_wW_z=e^{-\mathrm{i} \frac{\mathrm{Im}(w\cdot \overline{z})}{t}}W_{w+z},\ \ \  W_wk^t_z=e^{-\mathrm{i} \frac{\mathrm{Im}(w\cdot \overline{z})}{t}}k^t_{w+z}, \ \ \  \|W_zf\|_{F^p_t}=\|f\|_{F^p_t}.
\end{equation}
For more information on Weyl operators, one can consult Chapter 2 of \cite{Zhu}.

 Let us begin with the following lemma.
\begin{lem}\label{fockint}
Let $T$ be a bounded linear operator on $F^p_t$. Suppose that $h_0\in L^{\infty}(\mathbb{C}^n, dV)$ and $h_1,h_2\in F_t^1$, then we have
$$\sup_{u\in \mathbb C^n}|\langle TW_uh_1,W_uh_2 \rangle_t|\lesssim \|T\|~\|h_1\|_{F^1_t}\|h_2\|_{F^1_t}$$
and
$$\bigg|\int_{\mathbb{C}^n}h_0(u) \langle f,W_u h_1\rangle_t\langle W_uh_2,g\rangle_t dV(u)\bigg|\lesssim \|h_0\|_{\infty}\|h_1\|_{F^1_t}\|h_2\|_{F^1_t}\|f\|_{F^p_t}\|g\|_{F^q_t}$$
for all $f\in F^p_t$ and $g\in F^q_t$.
\end{lem}
\begin{proof}
First, we have
\begin{align*}
|\langle TW_uh_1,W_uh_2 \rangle_t| &\leqslant  \|T\|~ \|W_uh_1\|_{F^p_t}\|W_uh_2\|_{F^q_t}\\
&=\|T\|~\|h_1\|_{F^p_t}\|h_2\|_{F^q_t}\\
&\lesssim\|T\|~\|h_1\|_{F^1_t}\|h_2\|_{F^1_t},
\end{align*}
where the last inequality follows from \cite[Theorem 2.10]{Zhu}.  Recall that the identity operator can be written as
\begin{equation}\label{IT}
I=\frac{1}{(\pi t)^{n}}\int_{\mathbb C^n} (k_z^t\otimes k_z^t) dV(z)
\end{equation}
on $F_t^p$, where $$(k_z^t\otimes k_z^t)f=\langle f,k_z^t\rangle_t  k_z^t$$
for $f\in F_t^p$. Then we have
$$\langle f,W_u h_1\rangle_t=\frac{1}{(\pi t)^n}\int_{\mathbb{C}^n}\langle f,k_z^t\rangle_t\langle k_z^t, W_u h_1\rangle_t dV(z).$$
This gives us that
\begin{align*}
&\bigg|\int_{\mathbb{C}^n}h_0(u) \langle f,W_u h_1\rangle_t\langle W_uh_2,g\rangle_t dV(u)\bigg|\\
& \lesssim \int_{\mathbb{C}^n}|h_0(u)| \int_{\mathbb{C}^n}|\langle f,k_z^t\rangle_t|~|\langle k_z^t, W_u h_1\rangle_t| dV(z)
\int_{\mathbb{C}^n} |\langle W_uh_2,k^t_w\rangle_t| ~ |\langle k^t_w,g\rangle_t| dV(w) dV(u)\\
& \leqslant\|h_0\|_{\infty}\int_{\mathbb{C}^n} \int_{\mathbb{C}^n}|\langle f,k_z^t\rangle_t|
\bigg(\int_{\mathbb{C}^n} |\langle k_z^t, W_u h_1\rangle_t| ~ |\langle W_uh_2,k^t_w\rangle_t|dV(u)\bigg)|\langle k^t_w,g\rangle_t| dV(w)dV(z).
\end{align*}
Denoting
 $$F(z,w):=\int_{\mathbb{C}^n} |\langle k_z^t, W_u h_1\rangle_t| ~ |\langle W_uh_2,k^t_w\rangle_t|dV(u),$$
then H\"{o}lder's inequality gives
\begin{align*}
&\int_{\mathbb{C}^n} \int_{\mathbb{C}^n}|\langle f,k_z^t\rangle_t|
\bigg(\int_{\mathbb{C}^n} |\langle k_z^t, W_u h_1\rangle_t|~|\langle W_uh_2,k^t_w\rangle_t|dV(u)\bigg)|\langle k^t_w,g\rangle_t| dV(w)dV(z)\\
&=\int_{\mathbb{C}^n} \int_{\mathbb{C}^n}|\langle f,k_z^t\rangle_t| F(z,w)|\langle k^t_w,g\rangle_t| dV(w)dV(z)\\
&=\int_{\mathbb{C}^n} \int_{\mathbb{C}^n}|\langle f,k_z^t\rangle_t| F(z,w)^{\frac{1}{p}}F(z,w)^{\frac{1}{q}}dV(z)|\langle k^t_w,g\rangle_t| dV(w)\\
& \leqslant \bigg[\int_{\mathbb{C}^n}\Big( \int_{\mathbb{C}^n}|\langle f,k_z^t\rangle_t| F(z,w)^{\frac{1}{p}}F(z,w)^{\frac{1}{q}}dV(z)\Big)^pdV(w)\bigg]^{\frac{1}{p}}
\|g\|_{F^q_t}\\
& \leqslant \bigg[\int_{\mathbb{C}^n}\Big( \int_{\mathbb{C}^n}|\langle f,k_z^t\rangle_t|^p F(z,w)dV(z)\Big) \Big(\int_{\mathbb{C}^n}F(z,w)dV(z)\Big)^{\frac{p}{q}}dV(w)\bigg]^{\frac{1}{p}}
\|g\|_{F^q_t}\\
&\leqslant  \|f\|_{F^p_t}\|g\|_{F^q_t} \bigg[\sup_{z\in\mathbb C^n}\int_{\mathbb{C}^n}F(z,w)dV(w)\bigg]^{\frac{1}{p}}\bigg[\sup_{w\in \mathbb C^n}\int_{\mathbb{C}^n}F(z,w)dV(z)\bigg]^{\frac{1}{q}}.
\end{align*}
Furthermore, we have that
\begin{align*}
\sup_{z\in \mathbb C^n}\int_{\mathbb{C}^n}F(z,w)dV(w)
&\leqslant  \sup_{z\in \mathbb C^n} \int_{\mathbb{C}^n}\int_{\mathbb{C}^n} |\langle k_z^t, W_u h_1\rangle_t|~|\langle W_uh_2,k^t_w\rangle_t|dV(u)dV(w)\\
&=\sup_{z\in \mathbb C^n} \int_{\mathbb{C}^n}\int_{\mathbb{C}^n} |\langle k_{z-u}^t, h_1\rangle_t|~|\langle h_2,k^t_{w-u}\rangle_t|dV(u)dV(w)\\
&=\sup_{z\in \mathbb C^n} \int_{\mathbb{C}^n} |\langle k_{z-u}^t, h_1\rangle_t|\int_{\mathbb{C}^n}|\langle h_2,k^t_{w-u}\rangle_t|dV(w)dV(u)\\
&\lesssim \|h_1\|_{F^1_t}\|h_2\|_{F^1_t},
\end{align*}
where the last inequality follows from that
$$\int_{\mathbb{C}^n}|\langle \psi, k^t_{\lambda}\rangle_t|dV(\lambda)=(\pi t)^n \|\psi\|_{F_t^1}$$
for $\psi \in F_t^1$ and $\lambda\in \mathbb C^n$.  This yields that
$$\bigg|\int_{\mathbb{C}^n}h_0(u) \langle f,W_u h_1\rangle_t\langle W_uh_2,g\rangle_t dV(u)\bigg|
\lesssim \|h_0\|_{\infty}\|h_1\|_{F^1_t}\|h_2\|_{F^1_t}\|f\|_{F^p_t}\|g\|_{F^q_t}.$$
This completes the proof of the lemma.
\end{proof}
For each $h_0\in L^{\infty}(\mathbb{C}^n, dV)$ and $h_1,h_2\in F_t^1$, we define the  operator
$$\int_{\mathbb{C}^n}h_0(u) (W_uh_1\otimes W_uh_2) dV(u)$$
on $F_t^p$ by
$$\bigg\langle\int_{\mathbb{C}^n}h_0(u) (W_uh_1\otimes W_uh_2)dV(u)f,~g \bigg\rangle_t=\int_{\mathbb{C}^n}h_0(u) \langle f,W_uh_2\rangle_t
\langle W_uh_1 ,g \rangle_t dV(u),$$
where $f\in F_t^p$ and $g\in F_t^q$.
Note that Lemma \ref{fockint} guarantees  that this operator is bounded.

The following lemma is elementary, but we include a proof here for the sake of completeness.
\begin{lem}\label{fockdiff}
For any $r>0$ and $w,w'\in B(0,r)$, we have  that
\begin{equation*}
\|k^t_w-k^t_{w'}\|_{F^1_t}\leqslant C_r |w-w'|
\end{equation*}
for some positive constant $C_r$ depending  only on $r$.
\end{lem}
\begin{proof}
Using the definition of $W_w$ and (\ref{weyl}), we have
\begin{align*}
\|k^t_w-k^t_{w'}\|_{F^1_t}&=\Big\|W_w 1-W_w e^{\mathrm{i}\frac{\mathrm{Im}[w\cdot (\overline{w'-w})]}{t}}  k^t_{w'-w}\Big\|_{F^1_t}\\
&=\Big\|1-e^{\mathrm{i}\frac{\mathrm{Im}[w\cdot (\overline{w'-w})]}{t}}  k^t_{w'-w}\Big\|_{F^1_t}\\
&\leqslant  \Big|1-e^{\mathrm{i}\frac{\mathrm{Im}[w\cdot (\overline{w'-w})]}{t}}\Big| + \|1- k^t_{w'-w}\|_{F^1_t}\\
&\leqslant  \Big|1-e^{\mathrm{i}\frac{\mathrm{Im}[w\cdot (\overline{w'-w})]}{t}}\Big|+ \|1- k^t_{w'-w}\|_{F^2_{2t}}\\
& \leqslant  \Big|1-e^{\mathrm{i}\frac{\mathrm{Im}[w\cdot (\overline{w'-w})]}{t}}\Big| + \Big\|1- k_{2w'-2w}^{2t} e^{\frac{|w'-w|^2}{2t}}\Big\|_{F^2_{2t}} \\
& \leqslant \Big|1-e^{\mathrm{i}\frac{\mathrm{Im}[w\cdot (\overline{w'-w})]}{t}}\Big| + \Big|1-e^{\frac{|w'-w|^2}{2t}}\Big|+e^{\frac{|w'-w|^2}{2t}}\|1- k_{2w'-2w}^{2t}\|_{F^2_{2t}}\\
&=\Big|1-e^{\mathrm{i}\frac{\mathrm{Im}[w\cdot (\overline{w'-w})]}{t}}\Big| + \Big|1-e^{\frac{|w'-w|^2}{2t}}\Big|+e^{\frac{|w'-w|^2}{2t}}\Big|2-2e^{\frac{|w-w'|^2}{t}}\Big|\\
& \leqslant  C_r |w-w'|,
\end{align*}
where the third inequality follows from that $k_{w'-w}^t=k_{2w'-2w}^{2t} e^{\frac{|w'-w|^2}{2t}}$.
\end{proof}
Based on the  previous two lemmas, we are able to establish an integral representation for weakly localized operators on $F^p_t$.
\begin{thm}\label{fockintrep}
Let $A$ be a bounded linear operator on $F^p_t$. Then for each $r>0$, the mapping
$$w \mapsto \int_{\mathbb{C}^n}\langle AW_zk_0^t,W_{z}k^t_{w}\rangle_t (W_{z}k^t_{w})\otimes (W_zk_0^t) dV(z)= \int_{\mathbb{C}^n}\langle Ak_z^t,k^t_{z+w}\rangle_t (k^t_{z+w}\otimes k_z^t) dV(z)$$
is uniformly continuous and  uniformly  bounded from $B(0,r)$ to the set of bounded linear operators on $F^p_t$.
Moreover, the integral
$$\int_{B(0,r)}\int_{\mathbb{C}^n}\langle Ak_z^t,k^t_{z+w}\rangle_t \big(k^t_{z+w}\otimes k_z^t\big) dV(z)dV(w)$$
is convergent in the norm topology. Furthermore, if $A$ is  weakly localized on $F^p_t$, then
$$\frac{1}{(\pi t)^{2n}}\int_{B(0,r)}\int_{\mathbb{C}^n}\langle Ak_z^t,k^t_{z+w}\rangle_t \big(k^t_{z+w}\otimes k_z^t\big) dV(z)dV(w)$$
converges to $A$ in norm as $r\rightarrow\infty$.
\end{thm}
\begin{proof}
We obtain by  (\ref{weyl}) that
$$
\int_{\mathbb{C}^n}\langle AW_zk_0^t,W_{z}k^t_{w}\rangle_t \big(W_{z}k^t_{w}\otimes W_zk_0^t\big) dV(z)= \int_{\mathbb{C}^n}\langle Ak_z^t,k^t_{z+w}\rangle_t \big(k^t_{z+w}\otimes k_z^t\big) dV(z).$$
Then the first conclusion follows from Lemmas \ref{fockint} and  \ref{fockdiff}. Thus the integral
$$\int_{B(0,r)}\int_{\mathbb{C}^n}\langle Ak_z^t,k^t_{z+w}\rangle_t \big(k^t_{z+w}\otimes k_z^t\big) dV(z)dV(w)$$
is convergent in the norm topology.

Let $A$ be  weakly localized on $F^p_t$. For any $f\in F^p_t$ and $g\in F^q_t$, we have
\begin{align*}
\langle Af,g\rangle_t&=\frac{1}{(\pi t)^n}\int_{\mathbb{C}^n}\langle Af,k_w^t\rangle_t \langle k_w^t,g\rangle_t dV(w)\\
&=\frac{1}{(\pi t)^n}\int_{\mathbb{C}^n}\langle f,A^* k_w^t\rangle_t \langle k_w^t,g\rangle_t dV(w)\\
&=\frac{1}{(\pi t)^{2n}}\int_{\mathbb{C}^n}\int_{\mathbb{C}^n}\langle f,k_z^t\rangle_t \langle k_z^t ,A^* k_w^t\rangle_t \langle k_w^t,g\rangle_t dV(z)dV(w)\\
&=\frac{1}{(\pi t)^{2n}}\int_{\mathbb{C}^n}\int_{\mathbb{C}^n}\langle f,k_z^t\rangle_t \langle Ak_z^t , k_{w+z}^t\rangle_t \langle k_{w+z}^t,g\rangle_t dV(w)dV(z).
\end{align*}
Using the same method as in the  proof of Lemma \ref{fockint}, we get
\begin{align*}
&\bigg|\langle Af,g\rangle_t-\Big\langle\frac{1}{(\pi t)^{2n}}\int_{B(0,r)}\int_{\mathbb{C}^n}\langle Ak_z^t,k^t_{z+w}\rangle_t (k^t_{z+w}\otimes k_z^t) dV(z)dV(w)f,g\Big\rangle_t\bigg|\\
&=\frac{1}{(\pi t)^{2n}}\int_{\mathbb C^n\setminus B(0, r)}\int_{\mathbb{C}^n}|\langle f, k_z^t\rangle_t|~|\langle Ak_z^t , k_{w+z}^t\rangle_t|~|\langle k_{w+z}^t,g\rangle_t| dV(z)dV(w)\\
&=\frac{1}{(\pi t)^{2n}}\int_{\mathbb{C}^n}\int_{\mathbb C^n\setminus B(z, r)}|\langle f,k_z^t\rangle_t| ~ |\langle Ak_z^t , k_{w}^t\rangle_t|~|\langle k_{w}^t,g\rangle_t| dV(w)dV(z)\\
& \lesssim \|f\|_{F^p_t}\|g\|_{F^q_t}\Big[\sup_{z\in \mathbb C^n}\int_{\mathbb C^n\setminus B(z, r)}|\langle Ak_z^t , k_{w}^t\rangle_t|dV(w)\Big]^{\frac{1}{p}}
\Big[\sup_{w\in \mathbb C^n}\int_{\mathbb{C}^n}|\langle Ak_z^t , k_{w}^t\rangle_t|dV(z)\Big]^{\frac{1}{q}}.
\end{align*}
Now the desired conclusion follows from the definition of a weakly localized operator.
\end{proof}

Next, we will establish a sufficient condition for  a weakly localized operator  to be in the space $\mathcal{T}_{lin}(\mathcal{J})$ via the Berezin transform. Before going further, we still need some preparations.

Let $L$ be a bilinear map from $F^1_t\times F^1_t $ to some  Banach space $\mathcal{B}$. Suppose that $L(f,g)$ is linear with respect to $f$ and conjugate linear with respect to $g$. We say $L$ is bounded if
$$\|L(f,g)\|_{\mathcal B}\lesssim \|f\|_{F^1_t}\|g\|_{F^1_t}$$
for all $f, g\in F_t^1$.  For any multi-index $a=(a_1,\cdots,a_n)$ with $a_j\geqslant  0$ and $z\in \mathbb{C}^n$, we denote
$$z^a=z_1^{a_1}\cdots z_n^{a_n}, \ \ \ \ \ \ \ \ \ \  a!=a_1!\cdots a_n!$$
and $|a|=a_1+a_2+\cdots+a_n$.

With the notations above, we have the following proposition.
\begin{prop}\label{fockberezin}
Let $L$ be a bounded bilinear map from $F^1_t\times F^1_t $ to a Banach space $\mathcal{B}$. Let
$\mathcal{B}_1$ be a closed subspace of $\mathcal{B}$. If $L(k^t_z,k^t_z)\in \mathcal{B}_1$ for any $z\in \mathbb{C}^n$, then $L(k^t_w,k^t_z)\in \mathcal{B}_1$ for all $z,w\in \mathbb{C}^n$.
\end{prop}
\begin{proof}
We only need to show that  $L(K^t_w,K^t_z)\in \mathcal{B}_1$
 if $$L(K^t_z,K^t_z)\in \mathcal{B}_1$$
 for all  $z,w\in B(0,r)$ with $r>0$.

For any multi-index $a=(a_1,\cdots,a_n)$, let $g_a(\xi)={\xi}^{a}$, which is in $F^1_t$.
We know that $K^t_z(\xi)$ has a series expansion
$$K^t_z(\xi)=\sum_{a}c_a \overline{z}^ag_a(\xi)$$
with $c_a>0$.  Noting that
$$\lim_{m\rightarrow \infty}\sum_{|a|\leqslant m}c_a \overline{z}^ag_a(\xi)=K^t_z(\xi)$$
and
$$\Big|K_z^t(\xi)-\sum_{|a|\leqslant  m}c_a \overline{z}^ag_a(\xi)\Big|\leqslant  e^{\frac{r|\xi|}{t}}+e^{\frac{r|\xi_1|+\cdots+r|\xi_n|}{t}},$$
we conclude  by the dominated convergence theorem that
$$\lim_{m\rightarrow\infty}\Big\|K_z^t-\sum_{|a|\leqslant  m}c_a \overline{z}^ag_a\Big\|_{F^1_t}=0$$
for each $z\in \mathbb C^n$. Thus, it is enough to show that $L(g_a,g_b)\in \mathcal{B}_1$ for any multi-indices $a$ and $b$. Let
$$K_{z,a}^t:=\frac{\partial^a K_{u}^t}{\partial \overline{u}^a}\Big|_{u=z}.$$
Since $$\frac{\partial^a K_{z}^t}{\partial \overline{z}^a}\Big|_{z=0}=c_a a! g_a,$$
it is sufficient  to show that $L(K_{0,a}^t,K_{0,b}^t)\in \mathcal{B}_1$ for any two  multi-indices $a$, $b$. Let us prove this by induction. But to get the induction argument to work, we need to show that  $L(K_{z,a}^t,K_{z,b}^t)\in \mathcal{B}_1$ for all  multi-indices $a$ and $b$, and $z\in \mathbb C^n$.

First, when $a=b=0$, we have $$L(K_{z,0}^t,K_{z,0}^t)=L(K_{z}^t,K_{z}^t)\in \mathcal{B}_1.$$ Suppose that
$L(K_{z,a}^t,K_{z,b}^t)\in \mathcal{B}_1$
if  $|a+b|\leqslant  m$. Now we are going to show that $L(K_{z,a}^t,K_{z,b}^t)\in \mathcal{B}_1$ if $|a+b|=m+1$. Without loss of generality, we may assume that
$a_1\geqslant 1$. Let $e=(1,0,\cdots,0)$. Then there is  a multi-index $a'$ such that $a'+e=a$. For any $s>0$ satisfying  $z+se\in B(0,r)$, we have
\begin{align*}
&\frac{1}{s}\Big[L(K_{z+se,a'}^t,K_{z+se,b}^t)-L(K_{z,a'}^t,K_{z,b}^t)\Big]\\
&=L\bigg(\frac{K_{z+se,a'}^t-K_{z,a'}^t}{s},K_{z+se,b}^t\bigg)+L\bigg(K_{z,a'}^t,\frac{K_{z+se,b}^t-K_{z,b}^t}{s}\bigg).
\end{align*}
By the dominated convergence theorem again, we have
$$\lim_{s\rightarrow0}\bigg\|\frac{K_{z+se,a'}^t-K_{z,a'}^t}{s}-K_{z,a}^t\bigg\|_{F^1_t}=0
 \ \ \ \ \ \text{ and } \ \ \ \ \ \lim_{s\rightarrow0}\bigg\| K_{z+se,b}^t-K_{z,b}^t\bigg\|_{F^1_t}=0,$$
to obtain
\begin{align*}
L(K_{z,a}^t,K_{z,b}^t)+L(K_{z,a'}^t,K_{z,b+e}^t)
&=\lim_{s\rightarrow0}L\bigg(\frac{K_{z+se,a'}^t-K_{z,a'}^t}{s},K_{z+se,b}^t\bigg)+
\lim_{s\rightarrow0}L\bigg(K_{z,a'}^t,\frac{K_{z+se,b}^t-K_{z,b}^t}{s}\bigg)\\
&=\lim_{s\rightarrow0}\frac{1}{s}\Big[L(K_{z+se,a'}^t,K_{z+se,b}^t)-L(K_{z,a'}^t,K_{z,b}^t)\Big]\in \mathcal{B}_1.
\end{align*}

On the other hand, we have
\begin{align*}
&\frac{1}{\mathrm{i}s}\Big[L(K_{z+\mathrm{i}se,a'}^t,K_{z+\mathrm{i}se,b}^t)-L(K_{z,a'}^t,K_{z,b}^t)\Big]\\
&=L\bigg(\frac{K_{z+\mathrm{i}se,a'}^t-K_{z,a'}^t}{\mathrm{i}s},K_{z+\mathrm{i}se,b}^t\bigg)-L\bigg(K_{z,a'}^t,\frac{K_{z+\mathrm{i}se,b}^t-K_{z,b}^t}{\mathrm{i}s}\bigg).
\end{align*}
Similarly, we have
$$L(K_{z,a}^t,K_{z,b}^t)-L(K_{z,a'}^t,K_{z,b+e}^t)\in \mathcal{B}_1.$$
Therefore, we obtain that
$$2L(K_{z,a}^t,K_{z,b}^t)=L(K_{z,a}^t,K_{z,b}^t)-L(K_{z,a'}^t,K_{z,b+e}^t)+L(K_{z,a}^t,K_{z,b}^t)+L(K_{z,a'}^t,K_{z,b+e}^t)$$
belongs to $\mathcal{B}_1$. This finishes the proof of Proposition \ref{fockberezin}.
\end{proof}

In view of Proposition \ref{fockberezin}, we obtain that $\big\langle AW_{(\cdot)}k^t_z,W_{(\cdot)}k^t_w\big\rangle_t\in \mathcal{J}$  when  the Berezin transform of the operator $A$ is in $\mathcal{J}$, where $W_{(\cdot)}$ is a Weyl operator.
\begin{cor}\label{fockbi}
Let $A$ be a bounded linear operator on $F^p_t$, $z,w\in \mathbb{C}^n$ and $\mathcal{J}$ be a translation invariant closed  subspace of $\mathrm{BUC}(\mathbb{C}^n)$. If the Berezin transform of $A$ is in $\mathcal{J}$, then
$$\big\langle AW_{(\cdot)}k^t_z,W_{(\cdot)}k^t_w\big\rangle_t\in \mathcal{J}.$$
In addition, if $h\in \mathcal{J}$, then
$$\int_{\mathbb{C}^n}h(u)(W_uk^t_w)\otimes (W_uk^t_z)dV(u)\in \mathcal{T}_{lin}(\mathcal{J}).$$
\end{cor}
\begin{proof}Since $\mathcal{J}$ is translation invariant and $\widetilde{A}\in \mathcal{J}$, we have
$$\big\langle AW_{(\cdot)}k^t_z,W_{(\cdot)}k^t_z\big\rangle_t=\big\langle Ak^t_{(\cdot)+z},k^t_{(\cdot)+z}\big\rangle_t=\widetilde{A}((\cdot)+z)\in \mathcal{J}.$$
According to Lemma \ref{fockint} and  Proposition \ref{fockberezin}, we get that
$$\big\langle AW_{(\cdot)}k^t_z,W_{(\cdot)}k^t_w\big\rangle_t\in \mathcal{J}.$$

Let $h$ be in $\mathcal{J}$. Then  for any $f\in F^p_t$ and $g\in F^q_t$, we have
\begin{align*}
\Big\langle \int_{\mathbb{C}^n}h(u)(W_uk^t_z)\otimes (W_uk^t_z)dV(u)f,g\Big\rangle_t&
=\int_{\mathbb{C}^n}h(u)\langle f, k^t_{z+u}\rangle_t\langle k^t_{z+u},g\rangle_t dV(u)\\
&=\int_{\mathbb{C}^n}h(u-z)\langle f, k^t_{u}\rangle_t\langle k_u^t,g\rangle_t dV(u)\\
&= (\pi t)^n\langle T_{\alpha_{-z}h}f,g\rangle_t,
\end{align*}
where the translation $\alpha_{-z}$ is defined in Section \ref{Intro}.  Since $\mathcal{J}$ is translation invariant, we have that $\alpha_{-z}h\in \mathcal{J}$ and
$$\int_{\mathbb{C}^n}h(u)(W_uk^t_z)\otimes (W_uk^t_z)dV(u)=(\pi t)^n T_{\alpha_{-z}h}\in \mathcal{T}_{lin}(\mathcal{J}).$$
Using Lemma \ref{fockint} and Proposition  \ref{fockberezin} again, we obtain
$$\int_{\mathbb{C}^n}h(u)(W_uk^t_w)\otimes (W_uk^t_z) dV(u)\in \mathcal{T}_{lin}(\mathcal{J}),$$
to complete the proof of this corollary.
\end{proof}
Combining Theorem \ref{fockintrep} with Corollary \ref{fockbi} yields that a weakly localized operator on $F^p_t$ belongs to $\mathcal{T}_{lin}(\mathcal{J})$ if its Berezin transform is in $\mathcal{J}$.
\begin{prop}\label{fockweak}
Let $A$ be a weakly localized operator on $F^p_t$ and $\mathcal{J}$ be a translation invariant closed subspace of $\mathrm{BUC}(\mathbb{C}^n)$. If the Berezin transform of $A$ is in $\mathcal{J}$, then
$$ A\in \mathcal{T}_{lin}(\mathcal{J}).$$
\end{prop}
\begin{proof}
Since $A$ is  weakly localized, we have by Theorem \ref{fockintrep} that
$$A=\lim_{r\rightarrow\infty}\frac{1}{(\pi t)^{2n}}\int_{B(0,r)}\int_{\mathbb{C}^n}\langle AW_zk_0^t,W_zk^t_{w}\rangle_t (W_zk^t_{w})\otimes (W_zk_0^t) dV(z)dV(w).$$
Now Corollary \ref{fockbi} gives us that
$$\int_{\mathbb{C}^n}\langle AW_zk_0^t,W_zk^t_{w}\rangle_t (W_zk^t_{w})\otimes (W_zk_0^t) dV(z)\in \mathcal{T}_{lin}(\mathcal{J}).$$
This completes the proof.
\end{proof}

Finally, we show in the following that the Berezin transform of the finite product of Toeplitz operators with symbols in $\mathcal{J}$ also belongs to $\mathcal{J}$.
\begin{prop}\label{fockproduct}
Let $\mathcal{J}$ be a translation invariant closed subalgebra of $\mathrm{BUC}(\mathbb{C}^n)$. Suppose that $\varphi_1,\cdots,\varphi_m\in \mathcal{J}$ and $A=T_{\varphi_1}T_{\varphi_2}\cdots T_{\varphi_m}$. Then the Berezin transform $\widetilde{A}$ is in $\mathcal{J}.$
Moreover, if $\mathcal{J}_1$ is a translation invariant closed ideal of $\mathcal{J}$, $\psi_1, \psi_2, \cdots, \psi_k\in \mathcal{J}_1$ and  $B=T_{\psi_1}T_{\psi_2}\cdots T_{\psi_k}$, then
$$\widetilde{AB}\in \mathcal{J}_1 \ \ \ \ \text{ and }\ \ \ \  \widetilde{BA}\in \mathcal{J}_1.$$
\end{prop}
\begin{proof}
For a single Toeplitz operator $T_{\varphi_1}$, using $W_{-z} T_{\varphi_1} W_z= T_{\alpha_z \varphi_1}$ we have
$$
\langle T_{\varphi_1}k_z^t,k_z^t\rangle_t=\langle W_{-z}T_{\varphi_1}W_{z}1,1\rangle_t=\langle T_{\alpha_z\varphi_1}1,1\rangle_t=\int_{\mathbb{C}^n}\alpha_z\varphi_1(\xi)d\mu_t(\xi).
$$
Since $\varphi_1\in \mathcal{J}$ and $\mathcal{J}$ is translation invariant, we have that $\alpha_{(\cdot)}\varphi_1(\xi)\in \mathcal{J}$ for any $\xi\in \mathbb{C}^n$, and moreover,
$$\xi \mapsto  \alpha_{(\cdot)}\varphi_1(\xi)$$
is uniformly continuous and uniformly bounded from $\mathbb{C}^n$ to $\mathcal{J}$ with respect to the $L^{\infty}$-norm.
Thus we have
$$\big\langle T_{\varphi_1}k_{(\cdot)}^t,k_{(\cdot)}^t\big\rangle_t=\int_{\mathbb{C}^n}\alpha_{(\cdot)}\varphi_1(\xi)d\mu_t(\xi)\in \mathcal{J}.$$

To show that the Berezin transform of the product of $m$ Toeplitz operators belongs to $\mathcal{J}$, we suppose that the conclusion holds for $ m\leqslant k-1$. Then we are going to prove the conclusion  for the case of  $m=k$. Noting that $T_{\varphi_1}^*=T_{\overline{\varphi_1}}$, we have
by (\ref{IT})  that
\begin{align}\label{product}
\begin{split}
\langle T_{\varphi_1}T_{\varphi_2}\cdots T_{\varphi_k} k_z^t,k_z^t\rangle_t
&=\langle T_{\varphi_2}\cdots T_{\varphi_k} k_z^t,T_{\overline{\varphi_1}}k_z^t\rangle_t\\
&=\frac{1}{(\pi t)^n}\int_{\mathbb{C}^n} \langle T_{\varphi_2}\cdots T_{\varphi_k} k_z^t,k^t_{z+w}\rangle_t
\langle k^t_{z+w},T_{\overline{\varphi_1}}k_z^t\rangle_t dV(w)\\
&=\frac{1}{(\pi t)^n}\int_{\mathbb{C}^n} \langle T_{\varphi_2}\cdots T_{\varphi_k} W_zk_0^t,W_zk^t_{w}\rangle_t
\langle T_{\varphi_1}W_z k^t_{w},W_zk_0^t\rangle_t dV(w).
\end{split}
\end{align}

For $T_{\varphi_1}$ and the product $T_{\varphi_2}\cdots T_{\varphi_k}$,  we have by the induction hypothesis that
$$\big\langle T_{\varphi_1}W_{(\cdot)} k^t_{w},W_{(\cdot)}k_w^t\big\rangle_t \in \mathcal{J} \ \ \ \text{ and }\ \ \
 \big\langle T_{\varphi_2}\cdots T_{\varphi_k} W_{(\cdot)}k_w^t,W_{(\cdot)} k^t_{w}\big\rangle_t\in \mathcal{J} $$
for any $w\in \mathbb{C}^n$.  Since $\mathcal{J}$ is a translation invariant subalgebra, we obtain  by  Corollary \ref{fockbi} that
$$\big\langle T_{\varphi_2}\cdots T_{\varphi_k} W_{(\cdot)}k_0^t,W_{(\cdot)} k^t_{w}\big\rangle_t
\big\langle T_{\varphi_1}W_{(\cdot)} k^t_{w},W_{(\cdot)}k_0^t\big\rangle_t\in \mathcal{J}.$$
Using that $W_z^* T_\varphi W_z= T_{\alpha_z \varphi}$ on $F_t^p$ for $\varphi\in L^\infty(\mathbb C^n, dV)$, we have
\begin{align*}
&\big\|\big\langle T_{\varphi_2}\cdots T_{\varphi_k} W_{(\cdot)}k_0^t,W_{(\cdot)} k^t_{w}\big\rangle_t
\big\langle T_{\varphi_1}W_{(\cdot)} k^t_{w},W_{(\cdot)}k_0^t\big\rangle_t\big\|_{\infty}\\
&=\sup_{z\in \mathbb C^n} \big|\langle T_{\varphi_2}\cdots T_{\varphi_k} W_{z}k_0^t,W_{z} k^t_{w}\rangle_t
\langle T_{\varphi_1}W_{z} k^t_{w},W_{z}k_0^t\rangle_t\big|\\
&=\sup_{z\in \mathbb C^n} \big|\langle T_{\alpha_z \varphi_2}\cdots T_{\alpha_z \varphi_k} 1, k^t_{w}\rangle_t
\langle k^t_{w},T_{\overline{\alpha_z \varphi_1}}1 \rangle_t\big|\\
&\leqslant \sup_{z\in \mathbb C^n} \| T_{\alpha_z \varphi_2}\cdots T_{\alpha_z \varphi_k} 1\|_{F^2_t} \| k^t_{w}\|_{F^2_t}
|\langle k^t_{w},\overline{\alpha_z \varphi_1} \rangle_t|\\
&\leqslant\|\varphi_2\|_{\infty}\cdots\|\varphi_k\|_{\infty}  \sup_{z\in \mathbb C^n}\int_{\mathbb{C}^n}|\alpha_z\varphi_1(\xi)|\cdot\Big|e^{\frac{\xi\cdot \overline{w}}{t}}
e^{-\frac{|w|^2}{2t}}\Big| e^{-\frac{|\xi|^2}{t}}dV(\xi)\\
&\leqslant\|\varphi_1\|_{\infty} \|\varphi_2\|_{\infty}\cdots\|\varphi_k\|_{\infty}e^{-\frac{|w|^2}{4t}},
\end{align*}
where the last inequality comes from \cite[Corollary 2.5]{Zhu}.  By  Lemmas \ref{fockint} and \ref{fockdiff},  the mapping
$$w \mapsto \big\langle T_{\varphi_2}\cdots T_{\varphi_k} W_{(\cdot)}k_0^t,W_{(\cdot)}k^t_{w}\big\rangle_t
\big\langle T_{\varphi_1}W_{(\cdot)} k^t_{w},W_{(\cdot)}k_0^t\big\rangle_t$$
is uniformly continuous from each  compact subset of $\mathbb{C}^n$ to $\mathcal{J}$. Thus (\ref{product}) gives that
$$\widetilde{(T_{\varphi_1}T_{\varphi_2}\cdots T_{\varphi_k})}(\cdot)=\big\langle T_{\varphi_1}T_{\varphi_2}\cdots T_{\varphi_k}k_{(\cdot)}^t,k_{(\cdot)}^t\big\rangle_t\in \mathcal{J}.$$
This implies that $\widetilde{A}\in \mathcal{J}$ when $A=T_{\varphi_1}T_{\varphi_2}\cdots T_{\varphi_m}$ with $m\geqslant 1$, as desired.

By the definition of the Berezin transform of $AB$, we have
\begin{align*}
&\langle AB k_z^t,k_z^t\rangle_t=\frac{1}{(\pi t)^n}\int_{\mathbb{C}^n} \langle B W_zk_0^t,W_zk^t_{w}\rangle_t
\langle AW_z k^t_{w},W_zk_0^t\rangle_t dV(w).
\end{align*}
From the arguments above, it follows that
$$\big\langle A W_{(\cdot)}k_w^t,W_{(\cdot)}k^t_{0}\big\rangle_t\in\mathcal{J} \ \ \  \text{ and }\ \ \
\big\langle B W_{(\cdot)}k_0^t,W_{(\cdot)}k^t_{w}\big\rangle_t\in\mathcal{J}_1.$$
Since $\mathcal{J}_1$ is an ideal, we get that
$$\big\langle B W_{(\cdot)}k_0^t,W_{(\cdot)}k^t_{w}\big\rangle_t\big\langle A W_{(\cdot)}k_w^t,W_{(\cdot)}k^t_{0}\big\rangle_t \in\mathcal{J}_1,$$
to obtain  $\widetilde{AB}\in \mathcal{J}_1$.  Similarly, we can show that  $\widetilde{BA}\in \mathcal{J}_1$.
This completes the proof of Proposition \ref{fockproduct}.
\end{proof}

Now we are ready to present the proof of Theorem \ref{focktop}.
\begin{proof}[Proof of Theorem \ref{focktop}]
Let $\mathcal{J}$ be a translation invariant closed subalgebra of $\mathrm{BUC}(\mathbb{C}^n)$. Let us first show that
$$\mathcal{T}(\mathcal{J})=\mathcal{T}_{lin}(\mathcal{J}).$$
For $\varphi_1, \cdots, \varphi_m\in \mathcal{J}$, let $A=T_{\varphi_1}T_{\varphi_2}\cdots T_{\varphi_m}$. Proposition \ref{fockproduct} implies that  $\widetilde{A}\in \mathcal{J}.$ Using Proposition \ref{fockweak} and the fact that $A$ is  weakly localized, we have
$$A\in \mathcal{T}_{lin}(\mathcal{J}).$$
Observing that  $\mathcal{T}_{lin}(\mathcal{J})\subset\mathcal{T}(\mathcal{J})$ is obvious, so we have $\mathcal{T}(\mathcal{J})=\mathcal{T}_{lin}(\mathcal{J}).$

To complete the proof of Theorem \ref{focktop},  it remains to show that $\mathcal{T}(\mathcal{J}_1)$ is a two-sided ideal in $\mathcal{T}(\mathcal{J})$ if $\mathcal{J}_1$ is a translation invariant closed ideal of $\mathcal{J}$. To do so, we let $B=T_{\psi_1}T_{\psi_2}\cdots T_{\psi_k}$ with $\psi_1, \psi_2, \cdots,\psi_k\in \mathcal{J}_1$. Then we have by Proposition \ref{fockproduct} that
$$\widetilde{AB}\in \mathcal{J}_1 \ \ \ \text{ and }\ \ \ \widetilde{BA} \in \mathcal{J}_1.$$
Since $AB$ and $BA$ both are weakly localized operators, we deduce by  Proposition  \ref{fockweak} that
$$AB\in \mathcal{T}_{lin}(\mathcal{J}_1)=\mathcal{T}(\mathcal{J}_1)  \ \ \ \text{ and } \ \ \  BA\in \mathcal{T}_{lin}(\mathcal{J}_1)=\mathcal{T}(\mathcal{J}_1).$$
This completes the proof of Theorem \ref{focktop}.
\end{proof}

\section{Integral representations on the $p$-Bergman space}\label{Bergman-I-R}
The main purpose of this section is to establish an integral representation for $s$-weakly localized operators on the $p$-Bergman space $L_a^p$ with $1<p<\infty$. First, let us review some basic knowledge about the reproducing kernel for the Bergman space $L_a^2$, the M\"{o}bius transform and the Bergman metric on the unit ball $\mathbb B_n$.

 Recall that the reproducing kernel for the Bergman space $L_a^2$ is given by
$$K_{z}(w)=\frac{1}{(1-w\cdot \overline{z})^{n+1}},  \ \ \ z, w\in \mathbb B_n.$$
A simple  calculation shows that  $$c'_{p, q} (1-|z|^2)^{-\frac{n+1}{q}} \leqslant \|K_z\|_p \leqslant c_{p, q} (1-|z|^2)^{-\frac{n+1}{q}}$$
for some positive  constants $c_{p, q}$ and $c_{p, q}'$ depending only $p$ and $q$, where $\frac{1}{p}+\frac{1}{q}=1$. Letting
\begin{align}\label{k}
k_z^{(p)}:=(1-|z|^2)^{\frac{n+1}{q}}K_z,
\end{align}
then we have that $$c^{-1}_p\leqslant \|k^{(p)}_{z}\|_p\leqslant c_p$$
for some constant $c_p>0$ depending only on $p$. Recall that the Berezin transform $\widetilde{T}$ of a bounded linear operator $T$ on $L_a^p$ is defined by
$$\widetilde{T}(z)=\langle Tk_z,k_z\rangle, \ \ \ z\in \mathbb B_n,$$
where $k_z=k_z^{(2)}$ is the normalized reproducing kernel for $L^2_a$.

Let $\varphi_z$ be the M\"{o}bius transform of $\mathbb{B}_n$ that interchanges $0$ and $z$. Then we have
$$1- |\varphi_a(z)|^2=\frac{(1-|a|^2)(1-|z|^2)}{|1-\overline{a}\cdot z|^2},$$
see \cite[pages 25-26]{rudin} for the details. The \emph{Bergman metric} on $\mathbb{B}_n$ is defined by
$$\beta(z,w)=\frac{1}{2}\log\frac{1+|\varphi_z(w)|}{1-|\varphi_z(w)|}, \ \ \ z, w\in\mathbb B_n,$$
which is  M\"{o}bius invariant.  For each $z\in \mathbb B_n$ and $0<r<\infty$, the corresponding $\beta$-ball is given by
$$D(z,r)=\big\{w\in \mathbb B_n: \beta(z,w)< r \big\},$$
see pages 22-28 in \cite{Zhu2}.  Recall that the formula
$$d\lambda(z)=\frac{dv(z)}{(1-|z|^2)^{n+1}}$$
gives us  the standard M\"{o}bius-invariant measure on the unit ball. It is well-known that on $L_a^p$ we have
$$I=\int_{\mathbb B_n} (k_z\otimes k_z)d\lambda(z).$$

Recall that the pseudo-hyperbolic metric on $\mathbb B_n$ is defined by  $$ \rho(u,v) = |\varphi_{u}(v)|, \ \ \ \ u, v\in \mathbb B_n.$$
For the pseudo-hyperbolic metric, we have
\begin{align} \label{G}
\rho(\varphi_{u}(z),\varphi_{v}(z))\leqslant  \frac{G}{(1-|z|)^2 }\rho(u,v)
\end{align}
 holds for some  positive  constant $G$, which was obtained in \cite[Lemma 6.2]{suarez}.

The relationship between the Bergman metric and the  pseudo-hyperbolic metric on $\mathbb B_n$ is given by the following:
\begin{equation}\label{pseudohyper}
	\beta(u,v) =\frac{1}{2}\log\frac{1+|\rho(u,v)|}{1-|\rho(u,v)|} = \tanh^{-1}(\rho(u,v))
\end{equation}
see \cite[Corollary 1.22]{Zhu2} if required.

For a subset $\mathcal{I}$ of  $\mathrm{BUC}(\mathbb{B}_n)$, recall that  $\mathcal{I}$ is translation invariant if $\tau_zf\in \mathcal{I}$ for all $f\in \mathcal{I}$ and $z\in \mathbb{B}_n$, where $\tau_z$ is defined by (\ref{tau}):
$$(\tau_zf)(u)=f(\varphi_{u}(z)), \ \ \ u\in \mathbb B_n.$$ We will show in the next lemma that  $\mathrm{BUC}(\mathbb{B}_n)$ is translation invariant.

\begin{lem}\label{distantlemma}
Let  $z$, $u$ and $v$ be in $\mathbb{B}_n $.  If $\beta(u,v) < \tanh^{-1}\big[\frac{(1-|z|)^2}{G}\big]$, then we have
	$$\beta(\varphi_{u}(z),\varphi_{v}(z))\leqslant \tanh^{-1}\Big[\frac{G}{(1-|z|)^2 }\tanh\big(\beta(u,v)\big)\Big],$$
where $G$ is constant in (\ref{G}). As a consequence, $\mathrm{BUC}(\mathbb{B}_n)$ is translation invariant.
\end{lem}
\begin{proof}
Note that the assumption $\beta(u,v) < \tanh^{-1}\big[\frac{(1-|z|)^2}{G}\big]$ implies that $\frac{G}{(1-|z|)^2 }\rho(u,v)<1$. Thus $$\tanh^{-1}\Big[ \frac{G}{(1-|z|)^2 }\rho(u,v)\Big]$$ is well-defined.
It follows that
\begin{align*}
		\beta(\varphi_{u}(z),\varphi_{v}(z)) &= \tanh^{-1}(\rho(\varphi_{u}(z),\varphi_{v}(z))\leqslant \tanh^{-1}\Big[ \frac{G}{(1-|z|)^2 }\rho(u,v)\Big]\\
&=\tanh^{-1}\Big[\frac{G}{(1-|z|)^2 }\tanh\big(\beta(u,v)\big)\Big]
\end{align*}
since the funciton $\tanh^{-1}(x)$ is monotone increasing for $x\in (-1, 1)$. This proves the lemma.
\end{proof}

Let $z\in \mathbb B_n$, the operator $U_z$ is defined by
\begin{align}\label{Uf}
U_{z}f(w)=f(\varphi_{z}(w))k_z(w),  \ \ \ w\in \mathbb B_n, \  f\in L_a^p.
\end{align}
In particular, $U_{u}k_z=\eta(u,z)k_{\varphi_u(z)},$
where $u, z\in \mathbb B_n$  and  $\eta(u,z)=\frac{|1- u\cdot\overline{z}|^{n+1}}{(1- u\cdot \overline{z})^{n+1}}.$
Moreover, one can check readily that $U_z^*=U_z$ and
\begin{align}\label{U}
U_z^*T_\varphi U_z=T_{\varphi \circ \varphi_z}
\end{align}
 on $L_a^p$ for $\varphi \in L^\infty(\mathbb B_n, dv)$ and $z\in \mathbb B_n$.

The following definition of the \emph{$s$-weakly localized operator} was first introduced in \cite[Definition 1.4]{wick2014}, which plays an important role in the
characterization of Toeplitz algebras over the Bergman space $L_a^2$ \cite{Xia2015,Xia2017, Xia2018}.
\begin{defi}\label{s-weakly}
For any real number $s$ such that $0<s<\min\{p,q\}$, we say that a bounded linear operator $T$ on $L_a^p$ is $s$-weakly localized if it satisfies
$$\sup_{z\in \mathbb{B}_n}\int_{\mathbb{B}_n}|\langle Tk_z,k_w\rangle|\frac{\|K_z\|_2^{1-\frac{2s}{q(n+1)}}}{\|K_w\|_2^{1-\frac{2s}{q(n+1)}}}d\lambda(w)<\infty,$$
$$\sup_{z\in \mathbb{B}_n}\int_{\mathbb{B}_n}|\langle T^*k_z,k_w\rangle|\frac{\|K_z\|_2^{1-\frac{2s}{p(n+1)}}}{\|K_w\|_2^{1-\frac{2s}{p(n+1)}}}d\lambda(w)<\infty,$$
$$\lim_{r\rightarrow \infty}\sup_{z\in \mathbb{B}_n}\int_{\mathbb{B}_n\setminus D(z,r)}|\langle Tk_z,k_w\rangle|\frac{\|K_z\|_2^{1-\frac{2s}{q(n+1)}}}{\|K_w\|_2^{1-\frac{2s}{q(n+1)}}}d\lambda(w)=0,$$
$$\lim_{r\rightarrow \infty}\sup_{z\in \mathbb{B}_n}\int_{\mathbb{B}_n\setminus D(z,r)}|\langle T^*k_z,k_w\rangle|\frac{\|K_z\|_2^{1-\frac{2s}{p(n+1)}}}{\|K_w\|_2^{1-\frac{2s}{p(n+1)}}}d\lambda(w)=0.$$
The collection  of s-weakly localized operators on $L_a^p$ is denoted by $\mathcal{A}^{p}_s$.
\end{defi}

In the rest of this paper, we fix an $s$ such that $0<s<\min\{p,q\}$.  For an  $s$-weakly localized operator $T$ acting on $L_a^p$, we define
$$E_r(T,s)=\sup_{z\in \mathbb B_n}\int_{\mathbb B_n\setminus D(z,r)}|\langle Tk_z,k_w\rangle|\frac{\|K_z\|_2^{1-\frac{2s}{q(n+1)}}}{\|K_w\|_2^{1-\frac{2s}{q(n+1)}}}d\lambda(w)$$
and
$$E'_r(T,s)=\sup_{z\in \mathbb B_n}\int_{\mathbb B_n\setminus D(z,r)}|\langle T^*k_z,k_w\rangle|\frac{\|K_z\|_2^{1-\frac{2s}{p(n+1)}}}{\|K_w\|_2^{1-\frac{2s}{p(n+1)}}}d\lambda(w).$$
Define $D(z, 0)=\varnothing$ for $z\in \mathbb B_n$. Then we have that $E_0(T,s)<\infty$, $E'_0(T,s)<\infty$ and
\begin{equation*}\label{weakestimation}
\lim_{r\rightarrow\infty}E_r(T,s)=\lim_{r\rightarrow\infty}E'_r(T,s)=0.
\end{equation*}

The main theorem of this section is the following integral representation for $s$-weakly localized operators on the $p$-Bergman space $L_a^p$, which is parallel to Theorem \ref{fockintrep} in the previous section.
\begin{thm}\label{integralrepresantation}
Let $T$ be a bounded linear operator  on $L_a^p$ and  $r>0$. Then the mapping
$$v \mapsto  \int_{\mathbb B_n} \langle TU_uk_0,U_u k_{v}\rangle (U_u k_{v})\otimes (U_uk_0)d\lambda(u)$$
is uniformly continuous (with respect to the Bergman metric) and uniformly  bounded on $D(0, r)$. Moreover,  the integral
$$\int_{D(0,r)}\int_{\mathbb B_n} \langle TU_uk_0,U_u k_{v}\rangle (U_u k_{v})\otimes (U_uk_0) d\lambda(u) d\lambda(v)$$
is convergent in the norm topology. Furthermore, if $T\in \mathcal{A}^{p}_s$, then
$$\int_{D(0,r)}\int_{\mathbb B_n} \langle TU_uk_0,U_u k_{v}\rangle (U_u k_{v})\otimes (U_uk_0) d\lambda(u) d\lambda(v)$$
converges to $T$ in norm as $r\rightarrow \infty$.
\end{thm}

The following estimation related to the Bergman metric is useful for us to obtain the integral representation for $s$-weakly localized operators on $L_a^p$, see \cite[Lemma 2.20]{Zhu2} and \cite[Lemma 2.27]{Zhu2} if needed.
\begin{lem}
\label{discrete}
For any $R>0$ and any  $b\in \mathbb R$,  there exists a constant $C>0$ such that
$$\bigg|\frac{(1- z\cdot\overline{ u})^b}{(1- z\cdot \overline{v})^b}-1\bigg|\leqslant  C\beta(u,v),  \ \ \
\bigg|\frac{(1-|u|^2)^b}{(1-|v|^2)^b}-1\bigg|\leqslant C\beta(u,v),$$
$$C^{-1}(1-|u|^2)\leqslant 1-|v|^2\leqslant C(1-|u|^2)$$
and $$C^{-1}|1- z\cdot \overline{v}| \leqslant |1- z\cdot \overline{u}|\leqslant  C|1- z\cdot \overline{v}|$$
for all $z,u,v\in\mathbb{B}_n$ with $\beta(u,v)\leqslant R$.
\end{lem}
The next definition of \emph{separated set} in  the unit ball is quoted from \cite[Definition 2.1]{Xia2015}.
\begin{defi}\label{separated}
Let $\Gamma$ be a nonempty subset of $\mathbb{B}_n$. We say that $\Gamma$ is $\delta$-separated (or separated), if there exists a $\delta>0$ depending only  on $\Gamma$ such that
$ \beta(u,v)\geqslant \delta$ for all  $u\neq v$ in $\Gamma $.
\end{defi}

Then the proof of Theorem \ref{integralrepresantation}  begins with the following lemma.
\begin{lem}\label{bounded}
Let $h$ be in $L^\infty(\mathbb B_n, dv)$ and  $\Gamma$ be a $\delta$-separated  set. Let $\{h_{1,u}\}$ and $\{h_{2, u}\}$ be two families  of bounded analytic functions on $\mathbb B_n$. Then for each  $f\in L^p_a$ and $g\in L^q_a$, we have
$$\sum_{u\in\Gamma}\Big|h(u)\langle f, U_{u}h_{1,u} \rangle\langle U_{u}h_{2,u},g\rangle\Big|\leqslant
C \|h\|_{\infty} \sup_{u\in \mathbb{B}_n}\|h_{1,u}\|_{\infty}\sup_{u\in \mathbb{B}_n}\|h_{2, u}\|_{\infty}\|f\|_p \|g\|_q$$
and
$$\int_{\mathbb{B}_n}\Big|h(u)\langle f, U_{u}h_{1,u}\rangle\langle U_{u}h_{2,u},g\rangle\Big| d\lambda(u)\leqslant
 C' \|h\|_{\infty} \sup_{u\in \mathbb{B}_n}\|h_{1,u}\|_{\infty}\sup_{u\in \mathbb{B}_n}\|h_{2,u}\|_{\infty}\|f\|_p \|g\|_q,$$
where the constant $C$  depends only  on $\delta$ and $C'$ is an absolute constant.  Moreover,
for any bounded linear operator $T$ on $L^p_a$, we have
$$\sup_{u\in \mathbb{B}_n}|\langle TU_{u}h_{1,u},U_{u} h_{2,u}\rangle|\leqslant C'' \|T\| \sup_{u\in \mathbb{B}_n}\|h_{1,u}\|_{\infty} \sup_{u\in \mathbb{B}_n}\|h_{2,u}\|_{\infty},$$
where $C''$ is an absolute constant.
\end{lem}
\begin{proof}
Since $U_uh_{1,u}=(h_{1,u}\circ\varphi_u) k_u$ and $U_uh_{2,u}=(h_{2,u}\circ\varphi_u) k_u$ are bounded analytic functions,
$\langle f, U_{u}h_{1,u} \rangle $ and $\langle U_{u}h_{2,u},g\rangle$ are both well-defined. Then we have
\begin{align*}
&\sum_{u\in\Gamma}\Big|h(u)\langle f, U_{u}h_{1,u} \rangle\langle U_{u}h_{2,u},g\rangle\Big|\\
 &= \sum_{u\in\Gamma}\Big|h(u)\int_{\mathbb B_n} f(\xi) \langle K_\xi,U_{u}h_{1,u}\rangle dv(\xi) \int_{\mathbb B_n} \langle U_{u}h_{2,u},K_\zeta\rangle \overline{g(\zeta)}dv(\zeta)\Big| \\
& \leqslant  \|h\|_{\infty}\int_{\mathbb B_n}\int_{\mathbb B_n}|f(\xi)|\sum_{u\in\Gamma}|\langle K_\xi,U_{u}h_{1,u}\rangle \langle U_{u}h_{2,u},K_{\zeta}\rangle| dv(\xi)|g(\zeta)|dv(\zeta)
\end{align*}
for each $f\in L^p_a$ and $g\in L^q_a$. Denote
\begin{align}\label{F}
F(\xi,\zeta):=\sum_{u\in\Gamma}|\langle K_\xi,U_{u}h_{1,u}\rangle \langle U_{u}h_{2,u},K_\zeta\rangle|.
\end{align}
Then for any $b>0$, we have by  H\"{o}lder's inequality that
\begin{align*}
&\int_{\mathbb B_n}\int_{\mathbb B_n}|f(\xi)|F(\xi,\zeta) dv(\xi)|g(\zeta)|dv(\zeta)\\
&\leqslant \bigg[\int_{\mathbb B_n}\bigg(\int_{\mathbb B_n}|f(\xi)|F(\xi,\zeta) dv(\xi)\bigg)^pdv(\zeta)\bigg]^{\frac{1}{p}} \|g\|_q\\
&=\bigg[\int_{\mathbb B_n}\bigg(\int_{\mathbb B_n}|f(\xi)|F(\xi, \zeta)^{\frac{1}{p}}\frac{\|K_\zeta\|^b_2}{\|K_\xi\|^b_2} F(\xi,\zeta)^{\frac{1}{q}}\frac{\|K_\xi\|^b_2}{\|K_\zeta\|^b_2}dv(\xi)\bigg)^pdv(\zeta)\bigg]^{\frac{1}{p}} \|g\|_q.
\end{align*}
This  yields that
\begin{align*}
&\sum_{u\in\Gamma}\Big|h(u)\langle f, U_{u}h_{1,u} \rangle\langle U_{u}h_{2,u},g\rangle\Big|\\
&\leqslant \|h\|_\infty  \int_{\mathbb B_n}\int_{\mathbb B_n}|f(\xi)|F(\xi,\zeta) dv(\xi)|g(\zeta)|dv(\zeta)\\
& \leqslant \|h\|_\infty\|g\|_q\bigg[\int_{\mathbb B_n}\int_{\mathbb B_n}|f(\xi)|^pF(\xi,\zeta)\frac{\|K_\zeta\|^{bp}_2}{\|K_\xi\|^{bp}_2}dv(\xi)
\Big(\int_{\mathbb B_n} F(\xi,\zeta)\frac{\|K_\xi\|^{bq}_2}{\|K_\zeta\|^{bq}_2}dv(\xi)\Big)^\frac{p}{q}dv(\zeta)\bigg]^{\frac{1}{p}}\\
& \leqslant  \|h\|_\infty \|g\|_q\|f\|_p \bigg[\sup_{\zeta\in \mathbb B_n}\int_{\mathbb B_n} F(\xi,\zeta)\frac{\|K_\xi\|^{bq}_2}{\|K_\zeta\|^{bq}_2}dv(\xi)\bigg]^\frac{1}{q} \bigg[\sup_{\xi\in \mathbb B_n}\int_{\mathbb B_n} F(\xi,\zeta)\frac{\|K_\zeta\|^{bp}_2}{\|K_\xi\|^{bp}_2}dv(\zeta)
\bigg]^{\frac{1}{p}},
\end{align*}
where the second inequality is due to H\"{o}lder's inequality.

Next, we estimate the first term of the above inequality:
$$\sup_{\zeta\in \mathbb B_n}\int_{\mathbb B_n} F(\xi,\zeta)\frac{\|K_\xi\|^{bq}_2}{\|K_\zeta\|^{bq}_2}dv(\xi).$$
According to the definition of $U_u$ gives us that
\begin{align}\label{KU}
\begin{split}
&|\langle K_\xi, U_{u}h_{1,u}\rangle|= |\langle U_{u}K_\xi,h_{1,u}\rangle|=\bigg|\int_{\mathbb B_n} (U_{u}K_\xi)(w)\overline{h_{1,u}(w)}dv(w)\bigg|\\
&=\bigg|\int_{\mathbb B_n} K_\xi(\varphi_u(w))\frac{(1-|u|^2)^{\frac{n+1}{2}}}{(1-w\cdot \overline{u})^{n+1}}\overline{h_{1,u}(w)}dv(w)\bigg|.
\end{split}
\end{align}
Applying the change of variables formula in \cite[Proposition 1.13]{Zhu} to the last integral in (\ref{KU}),  we obtain that
\begin{align*}
|\langle K_\xi, U_{u}h_{1,u}\rangle|&=\bigg|\int_{\mathbb B_n} K_\xi(w)\frac{(1-|u|^2)^{\frac{n+1}{2}}}{[1-\varphi_u(w)\cdot \overline{u}]^{n+1}}\overline{h_{1,u}(\varphi_u(w))} \frac{(1-|u|^2)^{n+1}}{|1- w\cdot\overline{u}|^{2n+2}}dv(w)\bigg|\\
& \leqslant\int_{\mathbb B_n}  |K_\xi(w)|\frac{|1- w\cdot\overline{u}|^{n+1}}{(1-|u|^2)^{\frac{n+1}{2}}}\big|\overline{h_{1,u}(\varphi_u(w))}\big| \frac{(1-|u|^2)^{n+1}}{|1- w\cdot \overline{u}|^{2n+2}}dv(w)\\
& \leqslant \sup_{u\in \mathbb B_n} \|h_{1,u}\|_{\infty}\int_{\mathbb B_n} |K_\xi(w)| \frac{(1-|u|^2)^{\frac{n+1}{2}}}{|1-w\cdot \overline{u}|^{n+1}}dv(w).
\end{align*}
Similarly, we have
\begin{align*}
|\langle U_{u}h_{2,u},K_{\zeta}\rangle|\leqslant \sup_{u\in \mathbb B_n} \|h_{2,u}\|_{\infty}\int_{\mathbb B_n}  |K_\zeta(z)| \frac{(1-|u|^2)^{\frac{n+1}{2}}}{|1-z\cdot\overline{u}|^{n+1}}dv(z).
\end{align*}
Thus we obtain by (\ref{F}) that
\begin{align*}
F(\xi,\zeta)&\leqslant \sup_{u\in \mathbb B_n} \|h_{1,u}\|_{\infty}\sup_{u\in \mathbb B_n} \|h_{2,u}\|_{\infty}\times  \text{ I}(\xi, \zeta),
\end{align*}
where
$$\text{I}(\xi, \zeta):=\int_{\mathbb B_n}  \int_{\mathbb B_n} |K_\xi(w)|~|K_\zeta(z)| \sum_{u\in \Gamma} \frac{(1-|u|^2)^{n+1}}{|1- w\cdot\overline{u}|^{n+1}|1- z\cdot\overline{u}|^{n+1}}dv(z)dv(w).$$

To estimate the above  integral $\text{I}(\xi, \zeta)$, we need to deal with the summation first.  For any $u\in \Gamma$ and $\xi\in D(u,\frac{\delta}{2})$, we have by Lemma \ref{discrete} that
$$ \frac{(1-|u|^2)^{n+1}}{|1- z\cdot\overline{u}|^{n+1}|1- w\cdot\overline{u}|^{n+1}}
\leqslant C_1 \frac{(1-|\xi|^2)^{n+1}}{|1-z\cdot\overline{\xi}|^{n+1}|1- w\cdot\overline{\xi}|^{n+1}}
$$
for some constant $C_1>0$ depending only on $\delta$. This gives that
\begin{align*}
&\sum_{u\in\Gamma}\frac{(1-|u|^2)^{n+1}}{|1- z\cdot\overline{u}|^{n+1}|1- w\cdot\overline{u}|^{n+1}}\\
& \leqslant \sum_{u\in\Gamma}\frac{C_1}{\lambda(D(u,\frac{\delta}{2}))}\int_{D(u,\frac{\delta}{2})}\frac{(1-|\xi|^2)^{n+1}}{|1-z\cdot\overline{\xi}|^{n+1}|1- w\cdot\overline{\xi}|^{n+1}}d\lambda(\xi)\\
& \leqslant  C_2\int_{\mathbb B_n}\frac{dv(u)}{|1-z\cdot\overline{u}|^{n+1}|1- w\cdot\overline{u}|^{n+1}}\\
&= C_2 \int_{\mathbb B_n} |K_z(u)|~ |K_u(w)|dv(u),
\end{align*}
where $C_2$ is a positive constant depending only on $\delta$. Therefore,
$$\text{I}(\xi, \zeta)\leqslant C_2\int_{\mathbb B_n} \int_{\mathbb B_n} |K_\xi(w)|~|K_{\zeta}(z)|\int_{\mathbb B_n} |K_z(u)|~|K_u(w)|dv(u)dv(z)dv(w).$$
Denoting the integral on the right hand side of the above inequality by $\text{II}(\xi, \zeta)$, then we have that $\text{II}(\xi, \zeta)=\text{II}(\zeta, \xi)$ and
\begin{align*}
F(\xi,\zeta)&\leqslant C_2 \sup_{u\in \mathbb B_n} \|h_{1,u}\|_{\infty}\sup_{u\in \mathbb B_n} \|h_{2,u}\|_{\infty}\times \text{II}(\xi, \zeta).
\end{align*}
This gives us that
\begin{align*}
&\sup_{\zeta\in \mathbb B_n}\int_{\mathbb B_n} F(\xi,\zeta)\frac{\|K_\xi\|^{bq}_2}{\|K_\zeta\|^{bq}_2}dv(\xi)\\
&\leqslant C_2 \sup_{u\in \mathbb B_n} \|h_{1,u}\|_{\infty}\sup_{u\in \mathbb B_n}\|h_{2,u}\|_{\infty} \bigg( \sup_{\zeta\in \mathbb B_n}  \int_{\mathbb B_n} \text{II}(\xi, \zeta)\frac{\|K_\xi\|^{bq}_2}{\|K_\zeta\|^{bq}_2}dv(\xi)\bigg),
\end{align*}
where
\begin{align}\label{II}
\begin{split}
&\int_{\mathbb B_n} \text{II}(\xi, \zeta)\frac{\|K_\xi\|^{bq}_2}{\|K_\zeta\|^{bq}_2}dv(\xi)\\
&=\int_{\mathbb B_n} \int_{\mathbb B_n} \int_{\mathbb B_n}\int_{\mathbb B_n}|K_\xi(w)K_\zeta(z)K_z(u)K_u(w)|\frac{\|K_\xi\|^{bq}_2}{\|K_\zeta\|^{bq}_2}dv(\xi)dv(u)dv(z)dv(w).
\end{split}
\end{align}
Note that the  integrand in (\ref{II}) can be written as
\begin{align*}
&|K_\xi(w)K_\zeta(z)K_z(u)K_u(w)|\frac{\|K_\xi\|^{bq}_2}{\|K_\zeta\|^{bq}_2}\\
&=|K_z(\zeta)|\frac{\|K_z\|^{bq}_2}{\|K_\zeta\|^{bq}_2}|K_u(z)|\frac{\|K_u\|^{bq}_2}{\|K_z\|^{bq}_2} |K_w(u)|\frac{\|K_w\|^{bq}_2}{\|K_u\|^{bq}_2}   |K_\xi(w)|\frac{\|K_\xi\|^{bq}_2}{\|K_w\|^{bq}_2}.
\end{align*}
By \cite[Theorem 1.12]{Zhu2}, we have
\begin{align*}
\sup_{\zeta \in \mathbb B_n}\int_{\mathbb B_n}|K_z(\zeta)|\frac{\|K_z\|_2^{bq}}{\|K_\zeta\|_2^{bq}}dv(z)&=\sup_{\zeta\in \mathbb B_n}\int_{\mathbb B_n}\frac{(1-|z|^2)^{\frac{-bq(n+1)}{2}}(1-|\zeta|^2)^{\frac{bq(n+1)}{2}}}{|1-\zeta\cdot \overline{z}|^{n+1}}dv(z)\\
&=C_3<\infty
\end{align*}
if  $\frac{bq(n+1)}{2}<1$.
Now we choose $b>0$ such that $$\frac{bq(n+1)}{2}<1  \ \ \  \ \ \mathrm{and} \ \ \ \ \  \frac{bp(n+1)}{2}<1.$$
In this case, we obtain that the multiple integral in (\ref{II}) is less than some positive constant $C_3'.$
This leads to
$$\sup_{\zeta\in \mathbb B_n}\int_{\mathbb B_n} F(\xi,\zeta)\frac{\|K_{\xi}\|^{bq}_2}{\|K_{\zeta}\|^{bq}_2}dv(\xi)\leqslant C_4 \sup_{u\in \mathbb B_n} \|h_{1,u}\|_{\infty}\sup_{u\in \mathbb B_n}\|h_{2,u}\|_{\infty}.$$
Similarly, we have
$$\sup_{\xi\in \mathbb B_n}\int_{\mathbb B_n} F(\xi,\zeta)\frac{\|K_{\zeta}\|^{bp}_2}{\|K_{\xi}\|^{bp}_2}dv(\zeta) \leqslant C_5\sup_{u\in \mathbb B_n} \|h_{1,u}\|_{\infty}\sup_{u\in \mathbb B_n}\|h_{2,u}\|_{\infty}.$$
Combining the two inequalities above gives that
$$\sum_{u\in\Gamma}\Big|h(u)\langle f, U_{u}h_{1,u} \rangle\langle U_{u}h_{2,u},g\rangle\Big|\leqslant
C \|h\|_{\infty} \sup_{u\in \mathbb{B}_n}\|h_{1,u}\|_{\infty}\sup_{u\in \mathbb{B}_n}\|h_{2, u}\|_{\infty}\|f\|_p \|g\|_q$$
for some constant $C$ depending only on $\delta$, which is the first conclusion of this lemma.

Observe that the proof of  the above inequality implies that
$$\int_{\mathbb{B}_n}\Big|h(u)\langle f, U_{u}h_{1,u}\rangle\langle U_{u}h_{2,u},g\rangle\Big| d\lambda(u)\leqslant
 C' \|h\|_{\infty} \sup_{u\in \mathbb{B}_n}\|h_{1,u}\|_{\infty}\sup_{u\in \mathbb{B}_n}\|h_{2,u}\|_{\infty}\|f\|_p \|g\|_q$$
for some absolute constant  $C'>0$. This proves the second inequality in the lemma.

To obtain the last conclusion in our lemma, we first recall that
$$\|k_u\|_p\lesssim (1-|u|^2)^{\frac{n+1}{2}-\frac{n+1}{q}}, \ \ \ u\in \mathbb B_n.$$
Then  we have by (\ref{Uf}) that
\begin{align*}
|\langle TU_uh_{1,u}, U_uh_{2,u}\rangle|&\leqslant |\langle T h_{1,u}\circ\varphi_uk_u, h_{2,u}\circ\varphi_uk_u\rangle|\\
&\leqslant \|T\| ~ \|h_{1,u}\circ\varphi_uk_u\|_p\|h_{2,u}\circ\varphi_uk_u\|_q\\
&\leqslant \|T\|\sup_{u\in \mathbb B_n}\|h_{1,u}\|_{\infty}\sup_{u\in \mathbb B_n}\|h_{2,u}\|_{\infty}\|k_u\|_p\|k_u\|_q\\
& \lesssim \|T\|\sup_{u\in \mathbb B_n}\|h_{1,u}\|_{\infty}\sup_{u\in \mathbb B_n}\|h_{2,u}\|_{\infty}(1-|u|^2)^{\frac{n+1}{2}-\frac{n+1}{q}}
(1-|u|^2)^{\frac{n+1}{2}-\frac{n+1}{p}}\\
&= \|T\|\sup_{u\in \mathbb B_n}\|h_{1,u}\|_{\infty}\sup_{u\in \mathbb B_n}\|h_{2,u}\|_{\infty},
\end{align*}
as desired. This finishes the proof of Lemma \ref{bounded}.
\end{proof}

Using the same notations as in Lemma \ref{bounded},  $\Gamma$ denotes a separated set, $h$ is a bounded function on $\mathbb B_n$, $\{h_{1,u}\}$ and $\{h_{2,u}\}$ are two families of  bounded analytic functions on $\mathbb B_n$ with the property that
$$\sup_{u\in \mathbb B_n}\|h_{1,u}\|_{\infty}\sup_{u\in \mathbb B_n}\|h_{2,u}\|_{\infty}<\infty.$$
With the help of Lemma \ref{bounded}, now we can define two bounded linear operators
$$\sum_{u\in \Gamma} h(u)(U_{u}h_{1,u})\otimes (U_{u}h_{2,u}) \ \ \ \text{ and }\ \ \
\int_{\mathbb{B}_n} h(u)(U_{u}h_{1,u})\otimes (U_{u}h_{2,u})d\lambda(u)$$
as follows:
$$
\Big\langle \sum_{u\in \Gamma} h(u)(U_{u}h_{1,u})\otimes (U_{u}h_{2,u})f,g\Big\rangle
=\sum_{u\in \Gamma}h(u)\big\langle f, U_{u}h_{2,u}\big\rangle\big\langle U_{u}h_{1,u},g\big\rangle
$$
and
$$\Big\langle \int_{\mathbb{B}_n} h(u)(U_{u}h_{1,u})\otimes (U_{u}h_{2,u}) d\lambda(u) f,g\Big\rangle
=\int_{\mathbb{B}_n} h(u)\big\langle f, U_{u}h_{2,u}\big\rangle\big\langle U_{u}h_{1,u},g\big\rangle d\lambda(u),$$
where $f\in L_a^p$ and $g\in L_a^q$.  Noting that if  $h_{1,u}=h_{2,u}=1$, we have
$$\int_{\mathbb{B}_n} h(u)(U_{u}h_{1,u})\otimes (U_{u}h_{2,u})d\lambda(u)=\int_{\mathbb{B}_n} h(u)(k_u\otimes k_u) d\lambda(u),$$
and hence
\begin{align*}\label{Toepintegral}
\begin{split}
\Big\langle\int_{\mathbb{B}_n} h(u)(k_u\otimes k_u) d\lambda(u)f,g\Big\rangle
&=\int_{\mathbb{B}_n} h(u)\langle f,k_u\rangle  \langle k_u,g\rangle d\lambda(u)\\
&=\langle T_h f,g\rangle.
\end{split}
\end{align*}
This implies that the operator
$$\int_{\mathbb{B}_n} h(u)(k_u\otimes k_u) d\lambda(u)$$
is actually a Toeplitz operator with symbol $h$.

To establish Theorem \ref{integralrepresantation}, we need one more simple  lemma.
\begin{lem}\label{difference}
For each $u$ and $v$ in $D(0,r)$, we have
\begin{equation*}
\|k_u-k_v\|_{\infty}\leqslant C_r \beta(u,v),
\end{equation*}
where $C_r$ is a positive constant depending only on $r$.
\end{lem}
\begin{proof}
Let $r>0$, $z\in\mathbb{B}_n$ and $u, v\in D(0,r)$. We have by Lemma \ref{discrete} that
\begin{align*}
&|k_u(z)-k_v(z)|\\
&=\bigg|\frac{(1-|u|^2)^{\frac{n+1}{2}}}{(1- z\cdot\overline{u})^{n+1}}-\frac{(1-|v|^2)^{\frac{n+1}{2}}}{(1-z\cdot\overline{v})^{n+1}}\bigg|\\
&= \bigg|\frac{(1-|u|^2)^{\frac{n+1}{2}}(1- z\cdot\overline{v})^{n+1}-(1-|v|^2)^{\frac{n+1}{2}}(1-z\cdot\overline{u})^{n+1}}{|1- z\cdot\overline{u}|^{n+1}|1- z\cdot\overline{v}|^{n+1}}\bigg|\\
&\leqslant \bigg|\frac{(1-|u|^2)^{\frac{n+1}{2}}-(1-|v|^2)^{\frac{n+1}{2}}}{|1- z\cdot\overline{u}|^{n+1}}\bigg|+
(1-|v|^2)^{\frac{n+1}{2}}\bigg|\frac{(1- z\cdot\overline{v})^{n+1}-(1- z\cdot\overline{u})^{n+1}}{|1- z\cdot\overline{u}|^{n+1}|1- z\cdot\overline{v}|^{n+1}}\bigg|\\
& \leqslant  C_r' \big[\beta(u,v)|k_v(z)|+\beta(u,v)|k_u(z)|\big]\\
& \leqslant  C_r \beta(u,v),
\end{align*}
where the last inequality follows from that
$$|k_u(z)|+|k_v(z)|\leqslant 2^{\frac{n+1}{2}}\bigg[\frac{(1-|u|)^{\frac{n+1}{2}}}{(1-|u|)^{n+1}}+\frac{(1-|v|)^{\frac{n+1}{2}}}{(1-|v|)^{n+1}}\bigg]\leqslant C_r''$$
for $u, v\in D(0, r)$ and $z\in \mathbb B_n$. Here, $C_r', C_r$ and $C_r''$ denote the positive constants depending only on $r$.  This proves the lemma.
\end{proof}

We are now in position to prove the main result of this section.
\begin{proof}[Proof of Theorem \ref{integralrepresantation}]
First, combining Lemmas \ref{bounded} and \ref{difference}  gives us that the mapping
 $$v \mapsto  \int_{\mathbb B_n} \langle TU_uk_0,U_u k_{v}\rangle (U_u k_{v})\otimes (U_uk_0)d\lambda(u)$$
 is uniformly continuous  and  uniformly bounded on each $D(0, r)$. This yields that the integral
$$\int_{D(0,r)}\int_{\mathbb B_n} \langle TU_uk_0,U_u k_{v}\rangle (U_u k_{v})\otimes (U_uk_0)d\lambda(u) d\lambda(v)$$
is convergent in the norm topology.

Let $T$ be in $\mathcal{A}_s^p$.  Then the  M\"{o}bius-invariance of $d\lambda$ gives that
\begin{align*}
\langle Tf,g\rangle&=\int_{\mathbb B_n} \langle Tf,k_v\rangle \langle k_v,g\rangle d\lambda(v)\\
&=\int_{\mathbb B_n} \langle f,T^* k_v\rangle \langle k_v,g\rangle d\lambda(v)\\
&=\int_{\mathbb B_n} \int_{\mathbb B_n} \langle f,k_u\rangle \langle k_u,T^* k_v\rangle d\lambda(u)\langle k_v,g\rangle d\lambda(v)\\
&=\int_{\mathbb B_n} \int_{\mathbb B_n} \langle f,k_u\rangle \langle Tk_u, k_v\rangle d\lambda(u)\langle k_v,g\rangle d\lambda(v)\\
&=\int_{\mathbb B_n} \int_{\mathbb B_n} \langle f,k_u\rangle \langle Tk_u, k_{\varphi_u(v)}\rangle\langle k_{\varphi_u(v)},g\rangle d\lambda(u) d\lambda(v)\\
&=\int_{\mathbb B_n} \int_{\mathbb B_n} \langle f,U_uk_0\rangle \langle TU_uk_0, U_uk_{v}\rangle\langle U_uk_{v},g\rangle d\lambda(u)d\lambda(v)
\end{align*}
for each $f\in L^p_a$ and $g\in L_a^q$. This implies that
\begin{align*}
&\langle Tf,g\rangle-\Big\langle\int_{D(0,r)}\int_{\mathbb B_n} \langle TU_uk_0,U_u k_{v}\rangle (U_u k_{v})\otimes (U_uk_0) d\lambda(u) d\lambda(v)f,g\Big\rangle\\
&=\langle Tf,g\rangle-\int_{D(0,r)} \int_{\mathbb B_n} \langle f,U_uk_0\rangle \langle TU_uk_0, U_uk_{v}\rangle\langle U_uk_{v},g\rangle d\lambda(u) d\lambda(v)\\
&=\int_{\mathbb B_n\setminus D(0,r)} \int_{\mathbb B_n} \langle f,U_uk_0\rangle \langle TU_uk_0, U_uk_{v}\rangle\langle U_uk_{v},g\rangle d\lambda(u)d\lambda(v)\\
&=\int_{\mathbb B_n\setminus D(0,r)} \int_{\mathbb B_n} \langle f,k_u\rangle \langle Tk_u, k_{\varphi_u(v)}\rangle\langle k_{\varphi_u(v)},g\rangle d\lambda(u)d\lambda(v)\\
&=\int_{\mathbb B_n}\int_{\mathbb B_n\setminus D(0,r)} \langle f,k^{(q)}_u \rangle \big\langle Tk^{(p)}_u, k^{(q)}_{\varphi_u(v)}\big\rangle\big\langle k^{(p)}_{\varphi_u(v)},g\big\rangle d\lambda(v)d\lambda(u),
\end{align*}
where the last equality follows from (\ref{k}):
$$k_z^{(p)}(w)=\frac{(1-|z|^2)^{\frac{n+1}{q}}}{(1-w\cdot\overline{z})^{n+1}} \ \ \ \mathrm{and} \ \ \ k_z^{(q)}(w)=\frac{(1-|z|^2)^{\frac{n+1}{p}}}{(1-w\cdot\overline{z})^{n+1}} $$
for $z, w\in \mathbb B_n$. Therefore,
\begin{align*}
&\bigg|\langle Tf,g\rangle-\Big\langle\int_{D(0,r)}\int_{\mathbb B_n} \langle TU_uk_0,U_u k_{v}\rangle (U_u k_{v})\otimes (U_uk_0) d\lambda(u) d\lambda(v)f,g\Big\rangle\bigg|\\
&\leqslant \int_{\mathbb B_n}|\langle f,k^{(q)}_u\rangle|\bigg(\int_{\mathbb B_n\setminus D(u,r)} |\langle Tk^{(p)}_u, k^{(q)}_{v}\rangle|~|\langle k^{(p)}_{v},g\rangle| d\lambda(v)\bigg)d\lambda(u)\\
&\leqslant \bigg(\int_{\mathbb B_n}|\langle f,k^{(q)}_u\rangle|^pd\lambda(u)\bigg)^{\frac{1}{p}} \bigg(\int_{\mathbb B_n}\bigg|\int_{\mathbb B_n\setminus D(u,r)} \langle Tk^{(p)}_u, k^{(q)}_{v}\rangle \langle k^{(p)}_{v},g\rangle d\lambda(v)\bigg|^{q}d\lambda(u)\bigg)^{\frac{1}{q}}\\
&=\|f\|_p \bigg(\int_{\mathbb B_n}\Big|\int_{\mathbb B_n\setminus D(u,r)} \langle Tk^{(p)}_u, k^{(q)}_{v}\rangle\langle k^{(p)}_{v},g\rangle d\lambda(v)\Big|^{q}d\lambda(u)\bigg)^{\frac{1}{q}},
\end{align*}
where the last equality follows from that
$$\int_{\mathbb B_n} |\langle f,k^{(q)}_u\rangle|^p d\lambda(u)=\int_{\mathbb B_n} (1-|u|^2)^{n+1}|f(u)|^p d\lambda(u)=\|f\|^p_p.$$
Furthermore, we note that
\begin{align*}
&\int_{\mathbb B_n}\bigg|\int_{\mathbb B_n\setminus D(u,r)} \langle Tk^{(p)}_u, k^{(q)}_{v}\rangle\langle k^{(p)}_{v},g\rangle d\lambda(v)\bigg|^{q}d\lambda(u)\\
&\leqslant \int_{\mathbb B_n}\bigg[\int_{\mathbb B_n\setminus D(u,r)} |\langle Tk^{(p)}_u, k^{(q)}_{v}\rangle|^{\frac{1}{p}} \frac{\|K_u\|_2^b}{\|K_v\|_2^b}|\langle Tk^{(p)}_u, k^{(q)}_{v}\rangle|^{\frac{1}{q}}\frac{\|K_v\|_2^b}{\|K_u\|_2^b}|\langle k^{(p)}_{v},g\rangle| d\lambda(v)\bigg]^{q}d\lambda(u),
\end{align*}
where $$b=\frac{2(n+1)-2s}{pq(n+1)}.$$
Using H\"{o}lder's inequality, the above integral does not exceed $\text{I}'\times \text{II}'$, where
$$\text{I}':=\bigg(\sup_{u\in \mathbb B_n}\int_{\mathbb B_n\setminus D(u,r)} |\langle Tk^{(p)}_u, k^{(q)}_{v}\rangle| \frac{\|K_u\|_2^{bp}}{\|K_v\|_2^{bp}}d\lambda(v)\bigg)^{\frac{q}{p}}$$
and $$\text{II}':=\int_{\mathbb B_n} \int_{\mathbb B_n}|\langle Tk^{(p)}_u, k^{(q)}_{v}\rangle|\frac{\|K_v\|_2^{bq}}{\|K_u\|_2^{bq}}|\langle k^{(p)}_{v},g\rangle|^q d\lambda(v)d\lambda(u).$$
Using (\ref{k}) again, we observe that
$$
\text{I}' = \bigg(\sup_{u\in \mathbb B_n}\int_{\mathbb B_n\setminus D(u,r)} |\langle Tk_u, k_{v}\rangle| \frac{\|K_u\|_2^{bp+1-\frac{2}{q}}}{\|K_v\|_2^{bp+1-\frac{2}{q}}}d\lambda(v)\bigg)^{\frac{q}{p}}$$
and
$$\text{II}'= \int_{\mathbb B_n}\int_{\mathbb B_n}|\langle Tk_u, k_{v}\rangle|\frac{\|K_v\|_2^{bq+1-\frac{2}{p}}}{\|K_u\|_2^{bq+1-\frac{2}{p}}}|\langle k^{(p)}_{v},g\rangle|^q d\lambda(v)d\lambda(u).
$$
Elementary calculations give us that $$bp+1-\frac{2}{q}=1-\frac{2s}{q(n+1)}\ \ \  \ \text{and}\ \ \ \    bq+1-\frac{2}{p}=1-\frac{2s}{p(n+1)}.$$
Combining the above expressions for $\text{I}'$ and $\text{II}'$, we obtain that
\begin{align*}
&\int_{\mathbb B_n}\bigg|\int_{\mathbb B_n\setminus D(u,r)} \langle Tk^{(p)}_u, k^{(q)}_{v}\rangle\langle k^{(p)}_{v},g\rangle d\lambda(v)\bigg|^{q}d\lambda(u)\\
&\leqslant \Big(E_{r}(T,s)\Big)^{\frac{q}{p}}
\sup_{v\in \mathbb B_n}\int_{\mathbb B_n}|\langle k_u, T^*k_{v}\rangle|\frac{\|K_v\|^{bq+1-\frac{2}{p}}}{\|K_u\|^{bq+1-\frac{2}{p}}} d\lambda(u)
\int_{\mathbb B_n} |\langle k^{(p)}_{v},g\rangle|^q d\lambda(v)\\
&= \Big(E_{r}(T,s)\Big)^{\frac{q}{p}}E'_0(T,s)\|g\|_q^q.
\end{align*}
This yields
\begin{align*}
&\bigg|\langle Tf,g\rangle-\Big\langle\int_{D(0,r)}\int_{\mathbb B_n} \langle TU_uk_0,U_u k_{v}\rangle (U_u k_{v})\otimes (U_uk_0)d\lambda(u) d\lambda(v)f,g\Big\rangle\bigg|\\
& \leqslant \|f\|_p \bigg(\int_{\mathbb B_n}\Big|\int_{\mathbb B_n\setminus D(u,r)} \langle Tk^{(p)}_u, k^{(q)}_{v}\rangle\langle k^{(p)}_{v},g\rangle d\lambda(v)\Big|^{q}d\lambda(u)\bigg)^{\frac{1}{q}}\\
& \leqslant \Big(E_{r}(T,s)\Big)^{\frac{1}{p}}\Big(E'_0(T,s)\Big)^{\frac{1}{q}}\|f\|_p\|g\|_q.
\end{align*}
According to our assumption that $T\in \mathcal{A}_s^p$, we have
$$\lim_{r\rightarrow\infty}\Big(E_{r}(T,s)\Big)^{\frac{1}{p}}\Big(E'_0(T,s)\Big)^{\frac{1}{q}}=0,$$
to complete the proof of Theorem \ref{integralrepresantation}.
\end{proof}

\section{Toeplitz algebras over the $p$-Bergman space}\label{Bergman-T-A}
In the final section, we will study Toeplitz algebras with symbols in a translation invariant subalgebra  $\mathcal I\subset \mathrm{BUC}(\mathbb{B}_n)$ in terms of  the integral representation obtained in the previous section. More specifically, the following will be established in Theorem \ref{pbergthm} for the Toeplitz algebras
$\mathcal{T}^b[\mathrm{BUC}(\mathbb{B}_n)]$ and $\mathcal{T}^b[C_0(\mathbb{B}_n)]$ over  the $p$-Bergman space $L_a^p$:
$$ \mathcal{T}^b[\mathrm{BUC}(\mathbb{B}_n)]=\mathcal{T}^b_{lin}[\mathrm{BUC}(\mathbb{B}_n)]=\mathrm{clos}(\mathcal{A}_s^p)\text{\quad and\quad}
\mathcal{T}^b[C_0(\mathbb{B}_n)]=\mathcal{T}^b_{lin}[C_0(\mathbb{B}_n)]=\mathcal{K},$$
where $\mathrm{clos}(\mathcal{A}_s^p)$ denotes the closure of $\mathcal{A}_s^p$ with respect to the operator norm, and
$\mathcal{K}$ denotes the ideal of compact operators on $L^p_a$.
The basic strategy used to
prove the above conclusions  is similar to the strategy used in Section \ref{Fock} to prove the corresponding results for the case of $F_t^p$.

Let $H(\mathbb{B}_n)$ be the set of  functions which are holomorphic on the open unit ball $\mathbb B_n$.
Let $G$ be a bilinear map from $H(\mathbb{B}_n)\times H(\mathbb{B}_n)$ to a Banach space $\mathcal{B}$. We suppose that $G(f,g)$ is linear with respect to $f$ and conjugate linear with respect to $g$. We say that $G$ is bounded if
$$\|G(f,g)\|_{\mathcal{B}}\lesssim\|f\|_{\infty}\|g\|_{\infty}$$
for all $f, g\in H(\mathbb{B}_n)$.
\begin{prop}\label{berezin}
Let $G$ be a bounded bilinear map from $H(\mathbb{B}_n)\times H(\mathbb{B}_n)$ to some  Banach space $\mathcal{B}$. Let
$\mathcal{B}_1$ be a closed subspace of $\mathcal{B}$. If $G(k_z,k_z)\in \mathcal{B}_1$ for each $z\in \mathbb{B}_n$, then $G(k_w,k_z)\in \mathcal{B}_1$ for all $z, w\in \mathbb{B}_n$.
\end{prop}
\begin{proof} Note that $G(k_z,k_w)\in \mathcal{B}_1$ is equivalent to $G(K_z,K_w)\in \mathcal{B}_1$ for $z, w\in \mathbb B_n$. Thus, we only need to show that
$G(K_z,K_w)\in \mathcal{B}_1$ if $G(K_z,K_z)\in \mathcal{B}_1$.

Fix a multi-index $a=(a_1,a_2,\dots,a_n)$ with $a_j\geqslant 0$.  Let $e_a$ be the  function  on $\mathbb B_n$ of the form:
$$e_{a}(v)=v^a=v_1^{a_1}v_2^{a_2}\cdots v_n^{a_n}.$$
For any $r>0$ and $z\in D(0,r)$,  the reproducing kernel can be represented as:
$$K_z=\sum_{a}c_a \overline{z}^a e_a,$$
where each $c_a$ is a nonzero constant. Moreover, the summation  above is uniformly convergent on each $D(0, r)$, i.e.,
$$\lim_{m\rightarrow\infty}\sup_{z\in D(0,r)}\Big\|K_z-\sum_{|a|\leqslant m}c_a \overline{z}^a e_a\Big\|_{\infty}=0,$$
where $|a|=a_1+a_2+\cdots+a_n$ and $0<r<\infty$.

As $G$ is bounded, we only need to show that $G(e_a,e_b)\in \mathcal{B}_1$ for any two multi-indices $a$ and $b$. Denote
$$K_{z, a}:=\frac{\partial^{a} K_u}{\partial \overline{u}^a}\Big|_{u=z}.$$
 It is sufficient to show that $G(K_{0, a}, K_{0, b})\in \mathcal{B}_1$ for all multi-indices $a$ and $b$,
 since $K_{z, a}\big|_{z=0}=c_a a! e_a$.  In a similar way as in the proof of Proposition \ref{fockberezin}, we will use induction and show
that $G(K_{z, a},K_{z, b})\in \mathcal{B}_1$ for all  multi-indices $a$ and $b$, and $z\in \mathbb B_n$.

First,  we have by the assumption that
$$G(K_{z, 0}, K_{z, 0})\in \mathcal{B}_1.$$
For any positive integer $m$, we suppose that
$$G(K_{z, a}, K_{z, b})\in \mathcal{B}_1$$
for any $a$ and $b$ with $|a+b|\leqslant m.$  Let $a',b'$ be two indices such that $$|a'+b'|=m+1.$$ Without loss of generality, we may assume that
there exist multi-indices $a$ and $b$ such that
$$a'=a+e, \ \ \ \  b'=b,$$
where $e=(1,0,\cdots,0)$.  Since $|a+b|\leqslant m$, we have
$$G(K_{z, a}, K_{z, b})\in \mathcal{B}_1.$$

On one hand, we have
\begin{align*}
&\frac{1}{t}\Big(G(K_{z+te, a}, K_{z+te, b})-G(K_{z, a}, K_{z, b})\Big)\\
&=G\bigg(\frac{K_{z+te, a}-K_{z, a}}{t},K_{z+te, b}\bigg)+G\bigg(K_{z, a},\frac{K_{z+te, b}-K_{z, b}}{t}\bigg).
\end{align*}
Since
$$\lim_{t\rightarrow0}\Big\|\frac{K_{z+te, a}-K_{z, a}}{t}-K_{z, a+e}\Big\|_{\infty}=0 \ \ \ \text{ and } \ \ \ \lim_{t\rightarrow 0}\big\|K_{z+te, b}-K_{z, b}\big\|_{\infty}=0,$$
we have by the boundedness of $G$ that
\begin{align*}
&G(K_{z, a+e}, K_{z, b})+G(K_{z, a}, K_{z, b+e})\\
&=\lim_{t\rightarrow 0}G\Bigg(\frac{K_{z+te, a}-K_{z, a}}{t},K_{z+te, b}\bigg)+\lim_{t\rightarrow0}G\bigg(K_{z, a},\frac{K_{z+te, b}-K_{z, b}}{t}\bigg)\\
&=\lim_{t\rightarrow 0}\frac{1}{t}\Big[G(K_{z+te, a},K_{z+te, b})-G(K_{z, a},K_{z, b})\Big]\in\mathcal{B}_1.
\end{align*}

On the other hand,
\begin{align*}
&G(K_{z, a+e},K_{z, b})-G(K_{z, a}, K_{z, b+e)})\\
&=\lim_{t\rightarrow0}G\bigg(\frac{K_{z+\mathrm{i}te, a}-K_{z, a}}{\mathrm{i}t},K_{z+\mathrm{i}te, b}\bigg)
-\lim_{t\rightarrow0}G\bigg(K_{z, a},\frac{K_{z+\mathrm{i}te, b}-K_{z, b}}{\mathrm{i}t}\bigg)\\
&=\lim_{t\rightarrow0}\bigg[G\bigg(\frac{K_{z+\mathrm{i}te, a}-K_{z, a}}{\mathrm{i}t},K_{z+\mathrm{i}te, b}\bigg)
-G\bigg(K_{z, a},\frac{K_{z+\mathrm{i}te, b}-K_{z, b}}{\mathrm{i}t}\bigg)\bigg]\\
&=\lim_{t\rightarrow 0}\frac{1}{\mathrm{i}t}\bigg[G\bigg(K_{z+\mathrm{i}te, a}-K_{z, a},K_{z+\mathrm{i}te, b}\bigg)
+G\bigg(K_{z, a},K_{z+\mathrm{i}te, b}-K_{z, b}\bigg)\bigg]\\
&=\lim_{t\rightarrow0}\frac{1}{\mathrm{i}t}\Big[G\big(K_{z+\mathrm{i}te, a},K_{z+\mathrm{i}te, b}\big)-G\big(K_{z, a},K_{z, b}\big)\Big]\in\mathcal{B}_1.
\end{align*}
Thus we conclude that
$$2G(K_{z, a+e},K_{z, b})=G(K_{z, a+e},K_{z, b})+G(K_{z, a},K_{z, b+e})
+G(K_{z, a+e},K_{z, b})-G(K_{z, a}, K_{z, b+e})$$
belongs to $\mathcal{B}_1$.
This completes the proof of Proposition \ref{berezin}.
\end{proof}
Applying Proposition \ref{berezin} to the bilinear form $G(f, g)=\big \langle TU_{(\cdot)}f,U_{(\cdot)}g \big\rangle$, we  deduce that $\big \langle TU_{(\cdot)}k_z,U_{(\cdot)}k_w \big\rangle$ belongs to a translation invariant closed subalgebra of $\mathrm{BUC}(\mathbb{B}_n)$ if the Berezin transform of $T$ does.
\begin{cor}\label{biberezin}
Let $T$ be a bounded linear operator on $L_a^p$ and $\mathcal{I}$ be a translation invariant closed subspace of $\mathrm{BUC}(\mathbb{B}_n)$.
If the Berezin transform of $T$ is in $\mathcal{I}$, then
$$\big \langle TU_{(\cdot)}k_z,U_{(\cdot)}k_w \big\rangle\in \mathcal{I}$$
for all $z,w\in \mathbb{B}_n$.
\end{cor}
\begin{proof}
Since $\mathcal{I}$ is translation invariant and
$\widetilde{T}\in \mathcal{I},$ we have
$$ \big\langle TU_{(\cdot)}k_z,U_{(\cdot)}k_z \big\rangle=\big\langle Tk_{\varphi_{(\cdot)}(z)},k_{\varphi_{(\cdot)}(z)}\big\rangle= \widetilde{T}(\varphi_{(\cdot)}(z))=\tau_z(\widetilde{T})\in \mathcal{I}$$
for any $z\in \mathbb B_n$.
By Lemma \ref{bounded} and Proposition \ref{berezin}, we get that
$$\big\langle TU_{(\cdot)}k_z,U_{(\cdot)}k_w \big\rangle\in \mathcal{I}$$
for all $z, w\in \mathbb B_n$. This completes the proof.
\end{proof}

Since the Bergman metric is M$\mathrm{\ddot{o}}$bius invariant,  $\{ \varphi_{z}(u):  u\in \Gamma\}$ is separated for any $z$ when $\Gamma$ is separated. Although the set  $\{ \varphi_{u}(z):  u\in \Gamma\}$ may not be separated,  we have the following lemma in the case
that $\Gamma$ is $c$-separated (see Definition \ref{separated}).
\begin{lem}\label{separation}
Let $r>0$, $0\leqslant \rho<1$ and  $\Gamma$ be a $c$-separated set in $\mathbb B_n$. Then
there exist  a positive integer $m$ and a finite partition $\Gamma=\Gamma_1\cup \dots \cup\Gamma_m$ such that
 $\{\varphi_u(z): u\in \Gamma_j\}$ is $r$-separated for any $|z|\leqslant \rho$ and $j\in\{1,2,\cdots,m\}$.  Moreover, $m$ depends only on $r$, $\rho$ and $c$.
\end{lem}
\begin{proof}
Observing that this lemma is essentially comes from \cite[Lemma 2.2 (b)]{Xia2015}, we shall include a proof here for the sake of completeness. Let $R = r + \log\frac{1+\rho}{1-\rho}$ and  $$m = \bigg[\frac{\lambda\big(D(0,R+\frac{c}{2})\big)}{\lambda\big(D(0,\frac{c}{2})\big)}\bigg]+2,$$
where $[x]$ denotes the  greatest integer less than or equal to $x$.

From the proof of \cite[Lemma 2.2 (a)]{Xia2015}, we have that
	$$ \mathrm{card} \{v\in \Gamma: \beta(u,v) \leqslant R\} \leqslant \frac{\lambda\big(D(0,R+\frac{c}{2})\big)}{\lambda\big(D(0,\frac{c}{2})\big)} \leqslant\bigg[\frac{\lambda\big(D(0,R+\frac{c}{2})\big)}{\lambda\big(D(0,\frac{c}{2})\big)}\bigg]+1=m-1$$
 for any $u\in \Gamma$.
	
Next, we will show that there is a finite partition $\Gamma=\Gamma_1\cup \dots \cup\Gamma_m$ such that each $ \Gamma_j$ is $R$-separated.
If this is correct, then  we have that
	\begin{align*}
		\beta(\varphi_u(z) ,\varphi_v(z)) &\geqslant  \beta(u ,v) - \beta(\varphi_u(z) ,u) - \beta(v ,\varphi_v(z))= \beta(u ,v) - 2 \beta(z ,0) \\
		& \geqslant  R - \log\frac{1+|z|}{1-|z|}
		\geqslant  r + \log\frac{1+\rho}{1-\rho}- \log\frac{1+|z|}{1-|z|} \geqslant  r
	\end{align*}
for $\varphi_u(z)$ and $\varphi_v(z)$ with $u,v\in \Gamma_j$ and $|z|\leqslant \rho$. It follows that $\{\varphi_u(z): u\in \Gamma_j\}$ is $r$-separated for any $|z|\leqslant \rho$ and $1\leqslant j \leqslant m$.

Thus, it is sufficient to prove that there exists a  partition $\Gamma=\bigcup_{j=1}^m\Gamma_j$ such that $ \Gamma_j$ is $R$-separated for any $j$. To this end, we need a maximality argument.  Let us consider the set $G_i = \{ \Gamma^i_1,\cdots,\Gamma^i_m \}$, where $\Gamma_1^{i}, \Gamma_2^{i}, \cdots, \Gamma_m^{i}$ are subsets of $\Gamma$.  We say that \emph{$G_i$ is  $R$-separated}  if  $\Gamma_{\ell}^{i}\cap \Gamma_k^{i}=\varnothing$ when $\ell\neq k$ and each $\Gamma^i_{\ell}$ is $R$-separated. For any $G_{i_1}$ and $G_{i_2}$, if $\Gamma^{i_1}_{\ell} \subset \Gamma^{i_2}_{\ell}$ for all $\ell$, then we denote this by $G_{i_1}\prec G_{i_2}$. This is a partial order defined on  $\{G_i\}_i$ (the collection of all $R$-separated sets).

For a chain $\{G_i\}_{i\in \Lambda}$ in $\{G_i\}_{i}$, its maximal element is  given by
	$$\Big\{\bigcup_{i\in \Lambda}\Gamma_1^{i}, \bigcup_{i\in \Lambda}\Gamma_2^{i}, \cdots, \bigcup_{i\in \Lambda}\Gamma_m^{i} \Big\}.$$
According to Zorn's lemma,  $\{G_i\}_i$  has a maximal element and we denote it by $G_0 = \{\Gamma^0_1,\dots,\Gamma^0_m \}$. Then we must have that
	$$\bigcup_{j=1}^m\Gamma^0_j = \Gamma.$$
Otherwise, we can choose  $u\in\Gamma$ such that $u\notin \bigcup_{j=1}^m\Gamma^0_j$. If there exists a $v_j\in \Gamma^0_j$ such that $\beta(v_j, u)\leqslant  R$ for  any $j$, then we  would have
	$$\mathrm{card}\{ v\in \Gamma: \beta(v,u)\leqslant  R\} \geqslant  m,$$
	which is a contradiction. Thus there is an $\ell$ such that $\beta(v, u)> R$ for all $v\in \Gamma_{\ell}.$ Let $\Gamma'_{\ell} = \Gamma^{0}_{\ell} \bigcup \{u\}$ and
	$$G' =\big \{\Gamma^{0}_1, \cdots, \Gamma'_{\ell}, \cdots, \Gamma^{0}_m \big\}.$$
It follows that $G'$ is also $R$-separated, $G_0\prec G' $ and $G'\neq G_0$. But this contradicts the fact that $G_0$ is maximal.
Therefore, we conclude that there exists a partition $\Gamma=\bigcup_{j=1}^m\Gamma_j^0$ such that $ \Gamma_j^0$ is $R$-separated for every $j$.
This finishes the proof of Lemma \ref{separation}.
\end{proof}

In view of the above lemma, we are able to  construct a partition of the  unity on the unit ball. Let $c>0$, $l>0$ and $\Gamma$ be a $c$-separated set. Without loss of generality, we may assume that $0\in \Gamma$. Let $\Phi_0^l$ be a radial nonnegative smooth function on $\mathbb B_n$ such that $0\leqslant \Phi_0^l\leqslant 1$ and
\begin{equation}\label{supportfunction}
 \Phi^l_0(\xi)=\left\{
\begin{aligned}
&1 ,\text{ if }\xi\in D(0,l), \\
&0 ,\text{ if }\xi\notin D(0,2l).\\
\end{aligned}
\right.
\end{equation}
For $\zeta\in\Gamma$, we define $\Phi^l_\zeta$ by
\begin{align}\label{supportfunctionl}
\Phi^l_\zeta(\xi)=\Phi^l_0(\varphi_\zeta(\xi)),\ \ \ \ \xi\in \mathbb B_n.
\end{align}

If $\Gamma$ is a subset of the unit ball $\mathbb B_n$ such that $\{D(\zeta,\frac{c}{2}): \zeta \in \Gamma\}$ are mutually disjoint and $\{D(\zeta,c): \zeta\in \Gamma\}$ covers $\mathbb B_n$,  we  define
\begin{equation}\label{partition}
\Psi_\zeta(\xi):=\frac{\Phi^c_\zeta(\xi)}{\sum\limits_{\zeta\in\Gamma}\Phi^c_\zeta(\xi)}, \ \ \ \ \xi\in \mathbb B_n.
\end{equation}
Clearly,  $\Psi_\zeta(\varphi_{\zeta}(\xi))=\Psi_0(\xi)$, $\sum\limits_{\zeta\in \Gamma}\Psi_\zeta(\xi)=1$ and $\mathrm{supp}(\Psi_\zeta)\subset D(\zeta,2c)$.

Let $\mathcal{I}$ be a  translation invariant closed  subalgebra  of  $\mathrm{BUC}(\mathbb{B}_n)$. Then we define $\mathcal{I'}$ to be the subspace of $L^{\infty}(\mathbb B_n, dv)$ generated by
\begin{equation}\label{generatedspace}
\Big\{\sum_{u\in \Gamma}f(u) \Phi^l_{\varphi_u(z)}: l>0,\  \text{$\Gamma$ is a separated set, $z\in\mathbb{B}_n$ \ and \ $f\in\mathcal{I}$}\Big\}.
\end{equation}
Let $\psi =\varphi_{z}\circ\varphi_{u}$. Since $\psi(\varphi_{u}(z))=0$, we have by \cite[Theorem 2.2.5]{rudin} that there is a rotation $\mathcal R$ of $\mathbb{B}_n$  such that
$$\varphi_{z}\circ\varphi_{u} = \psi =\mathcal R \varphi_{\varphi_{u}(z)}.$$
Since $\Phi^l_0$ is radial, we obtain that
\begin{equation}\label{radial}
\Phi^l_{\varphi_{u}(z)}(\zeta)= \Phi^l_{0}[\varphi_{\varphi_{u}(z)}(\zeta)]=\Phi^l_{0}[\varphi_{z} (\varphi_{u}(\zeta))]=\Phi^l_{z}(\varphi_{u}(\zeta)).
\end{equation}
Let
$$C_0(\mathbb{B}_n)=\Big\{f\in \mathrm{BUC}(\mathbb{B}_n): \lim_{|z|\rightarrow1}f(z)=0\Big\}.$$

With the discussion above, we have the following result.
\begin{lem}\label{lemmaI}
Let $\mathcal{I}$  be a translation invariant closed  subalgebra of $\mathrm{BUC}(\mathbb{B}_n)$. Then the space $\mathcal{I'}$ defined by (\ref{generatedspace}) is contained in $\mathrm{BUC}(\mathbb{B}_n)$. Moreover, $\mathcal{I'}\subset C_0(\mathbb{B}_n)$ if $\mathcal{I}= C_0(\mathbb{B}_n)$.
\end{lem}

\begin{proof}
Let $l>0$, $z\in\mathbb{B}_n$ and  $\Gamma$ be a $c$-separated set. By Lemma \ref{separation}, for any $r>0$ there is
 a finite partition $\Gamma=\Gamma_1\cup \dots \cup\Gamma_m$ such that  $\{\varphi_u(z): u\in \Gamma_j\}$ is $r$-separated. Thus,
$$\sum_{u\in \Gamma}f(u) \Phi^l_{\varphi_u(z)}=\sum_{j=1}^m\sum_{u\in \Gamma_j}f(u) \Phi^l_{\varphi_u(z)}$$
for any $f\in\mathrm{BUC}(\mathbb{B}_n)$.
Now we  choose $r$ big enough such that the support sets of $\Phi^l_{\varphi_u(z)}$ are $c$-separated, i.e.,
$$\beta\Big(\mathrm{supp}\big(\Phi^l_{\varphi_u(z)}\big),\mathrm{supp}\big(\Phi^l_{\varphi_v(z)}\big)\Big)\geqslant c$$
for $u,v\in \Gamma_j$.

Next, we are going to show that  $\sum\limits_{u\in \Gamma_j}f(u) \Phi^l_{\varphi_u(z)}\in \mathrm{BUC}(\mathbb{B}_n)$. Let  $w_1$ and
$w_2$ be two points in $\mathbb{B}_n$ with $\beta(w_1, w_2)<c$. Since the support sets of $\Phi^l_{\varphi_u(z)}$ are $c$-separated for $u\in \Gamma_j$, there exists  $u_0\in \Gamma_j$ such that
$$\Phi^l_{\varphi_{u}(z)}(w_1) =\Phi^l_{\varphi_{u}(z)}(w_2) = 0   $$
for all $u\in \Gamma_j\backslash \{u_0\}.$ It follows  from (\ref{supportfunctionl}) that
\begin{align*}
	&\lim_{\beta(w_1,w_2)\rightarrow 0}\Big|\sum\limits_{u\in \Gamma_j}f(u) \Phi^l_{\varphi_u(z)}(w_1)-
	\sum\limits_{u\in \Gamma_j}f(u) \Phi^l_{\varphi_{u}(z)}(w_2)\Big| \\
	& \leqslant  \|f\|_{\infty}\lim_{\beta(w_1,w_2)\rightarrow 0}\Big| \Phi^l_{\varphi_{u_0}(z)}(w_1)- \Phi^l_{\varphi_{u_0}(z)}(w_2)\Big|\\
	 & =\|f\|_{\infty} \lim_{\beta(w_1,w_2)\rightarrow 0}\Big|\Phi^l_{0}[ \varphi_{\varphi_{u_0}(z)}(w_1)]-\Phi^l_{0}[ \varphi_{\varphi_{u_0}(z)}(w_2)]\Big|\\
	 & =\|f\|_{\infty} \lim_{d\rightarrow 0}\Big|\Phi^l_{0}[ \varphi_{\varphi_{u_0}(z)}(w_1)]-\Phi^l_{0}[\varphi_{\varphi_{u_0}(z)}(w_2)]\Big| \\
	&=0,
\end{align*}
where $$d=\beta\big(\varphi_{\varphi_{u_0}(z)}(w_1),\varphi_{\varphi_{u_0}(z)}(w_2)\big)=\beta(w_1, w_2)$$ and the last equality is due to that $\Phi^l_{0}$ is continuous with compact support. This means that $$\sum\limits_{u\in \Gamma_j}f(u) \Phi^l_{\varphi_u(z)}\in \mathrm{BUC}(\mathbb{B}_n),$$ to obtain
$$\sum_{u\in \Gamma}f(u) \Phi^l_{\varphi_u(z)}\in \mathrm{BUC}(\mathbb{B}_n).$$

Finally, we show that $\sum\limits_{u\in \Gamma_j}f(u) \Phi^l_{\varphi_u(z)}\in C_0(\mathbb{B}_n)$ for each $j\in \{1, 2, \cdots, m\}$ if $ f\in \mathcal{I}= C_0(\mathbb{B}_n)$.  Let  $w$ be a point in $\mathbb{B}_n$. If $w\notin \mathrm{supp}\big(\Phi^l_{\varphi_u(z)}\big)$ for any $u$, then  we have
	$$\sum\limits_{u\in \Gamma_j}f(u) \Phi^l_{\varphi_u(z)}(w) = 0.$$
Otherwise, suppose that $w\in \mathrm{supp}\big(\Phi^l_{\varphi_{u_w}(z)}\big)$ for some $u_w\in \Gamma_j$.  Then we have
\begin{align*}
	\Big|\sum\limits_{u\in \Gamma_j}f(u) \Phi^l_{\varphi_u(z)}(w)\Big| = \Big|f(u_w) \Phi^l_{\varphi_{u_w}(z)}(w)\Big|\leqslant  |f(u_w)|.
\end{align*}

Since $\beta(w,\varphi_{u_w}(z))<l$, we have $\beta(z,u_w)= \beta(0,\varphi_{u_w}(z) )>\beta(w,0)-l$. This yields that
$$\beta(u_w,0) >\beta(w,0)-l-\beta(z,0) .$$
Given an $\epsilon>0$.  Since  $f\in C_0(\mathbb B_n)$, there is a $\delta>0$ such that  $|f(u)|<\epsilon$ whenever $\beta(u,0)>\delta$.
If $\beta(w,0)>\delta+ l+ \beta(z,0)$, then we have that $\beta(u_{w},0)>\delta$ and
$$\Big|\sum\limits_{u\in \Gamma_j}f(u) \Phi^l_{\varphi_u(z)}(w)\Big| \leqslant |f(u_w)| <\epsilon. $$
This implies that
$$\sum_{u\in \Gamma_j}f(u) \Phi^l_{\varphi_u(z)}\in C_0(\mathbb{B}_n),$$
and so is  $\sum_{u\in \Gamma}f(u) \Phi^l_{\varphi_u(z)}$. Thus we conclude that $\mathcal{I'}\subset C_0(\mathbb{B}_n)$, to finish the proof of Lemma \ref{lemmaI}.
\end{proof}

Let $\mathcal{I}$ be a translation invariant closed subalgebra of $\mathrm{BUC}(\mathbb{B}_n)$ and  $\mathcal{I}'$ be defined in (\ref{generatedspace}).
Before showing that an $s$-weakly localized operator belongs to $\mathcal{T}^b_{lin}(\mathcal{I'})$ when its Berezin transform is in  $\mathcal{I}$, the following key proposition is required.

\begin{prop}\label{tensoroperartor}
Let $\mathcal{I}$ be a translation invariant closed subalgebra of $\mathrm{BUC}(\mathbb{B}_n)$ and $\Gamma$ be $c$-separated. For each $f\in\mathcal{I}$ and $z,w\in \Gamma$, we have
$$\sum_{u\in \Gamma}f(u) (U_u k_z)\otimes (U_uk_w) \in \mathcal{T}^b_{lin}(\mathcal{I'})$$ and
 $$\int_{\mathbb B_n} f(u) (U_u k_{z})\otimes (U_uk_w)d\lambda(u)\in \mathcal{T}^b_{lin}(\mathcal{I'}).$$
\end{prop}
\begin{proof}
Based on  Lemma $\ref{bounded}$ and Proposition $\ref{berezin}$, we need only to  show that
$$ \sum_{u\in \Gamma}f(u) (U_u k_z)\otimes (U_uk_z) \in \mathcal{T}^b_{lin}(\mathcal{I'}) $$ and
 $$\int_{\mathbb B_n} f(u) (U_u k_z)\otimes (U_uk_z) d\lambda(u)\in \mathcal{T}^b_{lin}(\mathcal{I'})$$
 for every $z\in \Gamma$.

Using $$U_uk_z=\frac{|1-u\cdot\overline{z}|^{n+1}}{(1-u\cdot\overline{z})^{n+1}}k_{\varphi_{u}(z)},$$ we have
$$ \sum_{u\in \Gamma} f(u) (U_u k_z)\otimes (U_uk_z)=\sum_{u\in \Gamma} f(u) k_{\varphi_{u}(z)}\otimes k_{\varphi_{u}(z)} $$
and
$$ \int_{\mathbb B_n} f(u) (U_u k_z)\otimes (U_uk_z) d\lambda(u)=\int_{\mathbb B_n} f(u) k_{\varphi_{u}(z)}\otimes k_{\varphi_{u}(z)}d\lambda(u).$$
By Lemmas  \ref{bounded} and \ref{difference}, we obtain that the following two mappings:
$$z \mapsto \sum_{u\in \Gamma} f(u) (U_u k_z)\otimes (U_uk_z)$$ and
$$ z \mapsto  \int_{\mathbb B_n} f(u) (U_u k_z)\otimes (U_uk_z) d\lambda(u)$$
are both uniformly continuous and uniformly bounded from  $D(0,r)$ to the set of bounded linear operators on $L_a^p$ in the norm topology, where $r>0$.

Let us first show that
$$\sum_{u\in \Gamma} f(u)k_{\varphi_{u}(z)}\otimes k_{\varphi_{u}(z)} \in \mathcal{T}^b_{lin}(\mathcal{I'}),$$
where
$f\in \mathcal{I}$ and  $\Gamma$ is $c$-separated.
To do this, let
$$a_l=\int_{\mathbb B_n}\Phi^l_{z}(\xi)d\lambda(\xi)$$
and
$$ A_l= \int_{\mathbb{B}_n}\sum_{u\in \Gamma}f(u)\frac{\Phi^l_{\varphi_{u}(z)}(\zeta)}{a_l}k_\zeta\otimes k_\zeta d\lambda(\zeta),$$
where the functions of the form  $\Phi^l_{w}$ are defined by (\ref{supportfunction}) and (\ref{supportfunctionl}). Since $A_l$ is  the Toeplitz operator with symbol $\sum\limits_{u\in \Gamma}f(u)\frac{\Phi^l_{\varphi_{u}(z)}(\zeta)}{a_l}$, we have that $A_l\in  \mathcal{T}_{lin}(\mathcal{I'}).$ Furthermore, we have by (\ref{supportfunction}) and (\ref{radial}) that
\begin{align*}
\int_{\mathbb{B}_n}\sum_{u\in \Gamma}f(u)\frac{\Phi^l_{\varphi_{u}(z)}(\zeta)}{a_l}k_\zeta\otimes k_\zeta d\lambda(\zeta)
=&\int_{\mathbb{B}_n}\sum_{u\in \Gamma}f(u)\frac{\Phi^l_{z}(\varphi_u(\zeta))}{a_l}k_\zeta\otimes k_\zeta d\lambda(\zeta)\\
=&\int_{\mathbb{B}_n}\sum_{u\in \Gamma}f(u)\frac{\Phi^l_{z}(\zeta)}{a_l}k_{\varphi_u(\zeta)}\otimes k_{\varphi_u(\zeta)} d\lambda(\zeta)\\
=&\int_{D(z,2l)}\frac{\Phi^l_{z}(\zeta)}{a_l}\sum_{u\in \Gamma}f(u)k_{\varphi_u(\zeta)}\otimes k_{\varphi_u(\zeta)} d\lambda(\zeta).
\end{align*}

We may assume  that  $0<l<1$. Then we have that $z, \zeta\in D\big(0,2+\tanh^{-1}(|z|)\big)$ if $\zeta\in D(z,2l)$. Using  Lemmas \ref{bounded} and \ref{difference} again, we obtain
$$\Big\|\sum_{u\in \Gamma}f(u)k_{\varphi_u(\zeta)}\otimes k_{\varphi_u(\zeta)}-
\sum_{u\in \Gamma}f(u)k_{\varphi_u(z)}\otimes k_{\varphi_u(z)}\Big\|\leqslant C_{|z|}l,$$
where $C_{|z|}$ is a constant depending only on $|z|$. This implies that
\begin{align*}
&\Big\|A_l-\sum_{u\in \Gamma}f(u)k_{\varphi_u(z)}\otimes k_{\varphi_u(z)}\Big\|\\
&=\Big\|\int_{D(z,2l)}\frac{\Phi^l_{z}(\zeta)}{a_l}\Big(\sum_{u\in \Gamma}f(u)k_{\varphi_u(\zeta)}\otimes k_{\varphi_u(\zeta)}-\sum_{u\in \Gamma}f(u)k_{\varphi_u(z)}\otimes k_{\varphi_u(z)}\Big) d\lambda(\zeta)\Big\|\\
& \leqslant  C'_{|z|}l
\end{align*}
for some constant $C'_{|z|}$ depending only on $|z|$. It follows  that
\begin{equation}\label{tensoralgebra}
\sum_{u\in \Gamma}f(u)k_{\varphi_u(z)}\otimes k_{\varphi_u(z)}\in\mathcal{T}^b_{lin}(\mathcal{I'})
\end{equation}
if $f\in \mathcal{I}$ and  $\Gamma$ is $c$-separated.

Next, we are going to show that
$$\int_{\mathbb B_n} f(u) k_{\varphi_{u}(z)}\otimes k_{\varphi_{u}(z)}d\lambda(u)\in\mathcal{T}^b_{lin}(\mathcal{I'}).$$  Let $\Gamma'$ be a subset of $\mathbb B_n$ such that $\{D(\zeta,\frac{c}{2}): \zeta\in \Gamma'\}$ are mutually disjoint and $\{D(\zeta,c):  \zeta\in \Gamma'\}$ covers $\mathbb B_n$. Recalling  that $\Psi_\zeta(\xi)$ is the function constructed  in (\ref{partition}), then we have that
\begin{align*}
&\Big\langle\int_{\mathbb B_n} f(u) k_{\varphi_{u}(z)}\otimes k_{\varphi_{u}(z)}d\lambda(u) h,g\Big\rangle\\
=&\int_{\mathbb B_n} f(u) \langle h ,k_{\varphi_{u}(z)}\rangle \langle k_{\varphi_{u}(z)} ,g\rangle d\lambda(u)\\
=&\int_{\mathbb B_n} \sum_{\zeta\in \Gamma'}\Psi_\zeta(u) f(u) \langle h ,k_{\varphi_{u}(z)}\rangle \langle k_{\varphi_{u}(z)} ,g\rangle d\lambda(u)\\
=&\int_{\mathbb B_n}\Psi_0(u) \sum_{\zeta\in \Gamma'}f(\varphi_\zeta(u)) \big\langle h ,k_{\varphi_{\varphi_\zeta(u)}(z)}\big\rangle \big\langle k_{\varphi_{\varphi_\zeta(u)}(z)} ,g\big\rangle d\lambda(u)
\end{align*}
for any $h\in L^p_a$ and $g\in L^q_a$.

By Lemma \ref{separation}, there is a partition $\Gamma'=\Gamma'_1\cup\cdots\cup \Gamma'_m$ such that
$ \{\varphi_\zeta(u): \zeta\in \Gamma_j'\}$ is $c$-separated for any $j\in \{1, 2, \cdots, m\}$.
Applying  Lemma \ref{separation} to each of the $c$-separated sets $\{\varphi_{\zeta}(u): \zeta\in \Gamma_j'\}$ (which is denoted by $G_j$),
we can get a partition $G_j=\bigcup_{\ell=1}^{N_j}G_{j, \ell}$ such that each $\{\varphi_{\varphi_\zeta(u)}(z): \varphi_\zeta(u)\in G_{j, \ell}\}$ is also $c$-separated. Let $\Gamma_{j,\ell}=\{\zeta: \varphi_\zeta(u)\in G_{j, \ell}\}$. Then we see that $\{\varphi_{\varphi_\zeta(u)}(z): \zeta\in \Gamma_{j, \ell}\}$ is  $c$-separated. This means that $$\Gamma=\bigcup_{j=1}^m \Gamma_j'=\bigcup_{j=1}^m\bigcup_{\ell=1}^{N_j}\Gamma_{j,\ell}.$$
From the above arguments we may assume that $\{\varphi_{\varphi_\zeta(u)}(z): \zeta\in \Gamma_{j}'\}$ is  $c$-separated for $j\in \{1, 2, \cdots, m\}$.

Furthermore, we have
\begin{align*}
&\Big\langle\int_{\mathbb B_n} f(u) k_{\varphi_{u}(z)}\otimes k_{\varphi_{u}(z)}d\lambda(u) h,g\Big\rangle\\
&=\sum_{j=1}^m\int_{\mathbb B_n}\Psi_0(u) \sum_{\zeta\in \Gamma_j'}f(\varphi_\zeta(u)) \big\langle h ,k_{\varphi_{\varphi_\zeta(u)}(z)}\big\rangle \big\langle k_{\varphi_{\varphi_\zeta(u)}(z)}, g\big\rangle d\lambda(u)\\
&=\sum_{j=1}^m\int_{D(0,2c)}\Psi_0(u) \Big\langle\sum_{\zeta\in \Gamma_j'}f(\varphi_\zeta(u)) k_{\varphi_{\varphi_\zeta(u)}(z)}\otimes k_{\varphi_{\varphi_\zeta(u)}(z)} h,g\Big\rangle d\lambda(u),
\end{align*}
where the last equality follows from that $\mathrm{supp}(\Psi_0)\subset D(0, 2c)$. Next,  we will show that the mapping
\begin{align}\label{m2}
u \mapsto  \sum_{\zeta\in \Gamma_j'}f(\varphi_\zeta(u)) k_{\varphi_{\varphi_\zeta(u)}(z)}\otimes k_{\varphi_{\varphi_\zeta(u)}(z)}
\end{align}
is uniformly continuous and uniformly bounded from $D(0, 2c)$ to the set of bounded linear operators on $L_a^p$ in the norm topology.

Let $u$ and $u'$ be two points in $D(0, 2c)$ with $\beta(u, u') < \tanh^{-1}\big(\frac{(1-|z|)^2}{2G}\big)$. Then we have by Lemma \ref{distantlemma} and the  M\"{o}bius invariance of $\beta$ that
\begin{align*}
\beta\big(\varphi_{\varphi_\zeta(u)}(z),\varphi_{\varphi_\zeta(u')}(z)\big)\leqslant \tanh^{-1}\Big[ \frac{G}{(1-|z|)^2 }\tanh\big(\beta(u,u')\big)\Big]
 \end{align*}
 for all $\zeta\in \mathbb B_n$. Letting $a(u,u',\zeta,z)=\varphi_{\varphi_{\varphi_\zeta(u)}(z)}(\varphi_{\varphi_\zeta(u')}(z)),$
 we have
$$\beta\big(a(u,u',\zeta,z),0\big) \leqslant \tanh^{-1}\Big[ \frac{G}{(1-|z|)^2 }\tanh\big(\beta(u,u')\big)\Big]$$
and $$\varphi_{\varphi_{\varphi_\zeta(u)}(z)}\big(a(u,u',\zeta,z)\big)=\varphi_{\varphi_\zeta(u')}(z).$$
Then it follows that
\begin{align*}
&\bigg\| \sum_{\zeta\in \Gamma_j'}f(\varphi_\zeta(u')) k_{\varphi_{\varphi_\zeta(u')}(z)}\otimes k_{\varphi_{\varphi_\zeta(u')}(z)}-\sum_{\zeta\in \Gamma_j'}f(\varphi_\zeta(u)) k_{\varphi_{\varphi_\zeta(u)}(z)}\otimes k_{\varphi_{\varphi_\zeta(u)}(z)}\bigg\|\\
&=\bigg\| \sum_{\zeta\in \Gamma_j'}\Big[f(\varphi_\zeta(u')) U_{\varphi_{\varphi_\zeta(u)}(z)}k_{a(u,u',\zeta,z)}\otimes U_{\varphi_{\varphi_\zeta(u)}(z)} k_{a(u,u',\zeta,z)}
-f(\varphi_\zeta(u))U_{\varphi_{\varphi_\zeta(u)}(z)}k_{0}\otimes U_{\varphi_{\varphi_\zeta(u)}(z)}k_0\Big]\bigg\|\\
&\leqslant \bigg\| \sum_{\zeta\in \Gamma_j'}\Big[[f(\varphi_\zeta(u'))-f(\varphi_\zeta(u))] U_{\varphi_{\varphi_\zeta(u)}(z)}k_{a(u,u',\zeta,z)}\otimes U_{\varphi_{\varphi_\zeta(u)}(z)} k_{a(u,u',\zeta,z)}\Big]\bigg\|\\
&\ \ +\bigg\| \sum_{\zeta\in \Gamma_j'}\Big[f(\varphi_\zeta(u))U_{\varphi_{\varphi_\zeta(u)}(z)}k_{a(u,u',\zeta,z)}\otimes U_{\varphi_{\varphi_\zeta(u)}(z)}  (k_{a(u,u',\zeta,z)}-k_0)\Big]\bigg\|\\
&\ \ +\bigg\| \sum_{\zeta\in \Gamma_j'}\Big[f(\varphi_\zeta(u))U_{\varphi_{\varphi_\zeta(u)}(z)}(k_{a(u,u',\zeta,z)}-k_0)\otimes U_{\varphi_{\varphi_\zeta(u)}(z)}k_0\Big]\bigg\|.
\end{align*}
As  $\{\varphi_{\varphi_\zeta(u)}(z): \zeta\in \Gamma_{j}'\}$ is  $c$-separated, we have that
\begin{align}\label{2}
\begin{split}
	&\bigg\|\sum_{\zeta\in \Gamma_j'}f(\varphi_\zeta(u))U_{\varphi_{\varphi_\zeta(u)}(z)}(k_{a(u,u',\zeta,z)}-k_0)\otimes U_{\varphi_{\varphi_\zeta(u)}(z)}k_0\bigg\|\\
	&= \sup_{g_1\in L_a^p, \|g_1\|_p = 1 \atop { g_2\in L_a^q, \|g_2\|_q=1}}\Big| \Big\langle\sum_{\zeta\in \Gamma_j'}f(\varphi_\zeta(u))U_{\varphi_{\varphi_\zeta(u)}(z)}(k_{a(u,u',\zeta,z)}-k_0)\otimes U_{\varphi_{\varphi_\zeta(u)}(z)}k_0 g_1, g_2\Big\rangle\Big|\\
	 & \leqslant  \sup_{g_1\in L_a^p, \|g_1\|_p = 1 \atop { g_2\in L_a^q, \|g_2\|_q=1}}\sum_{\zeta\in \Gamma_j'}|f(\varphi_\zeta(u))|~
	\Big|\Big\langle U_{\varphi_{\varphi_\zeta(u)}(z)}(k_{a(u,u',\zeta,z)}-k_0), g_2 \Big\rangle\Big|~
	\Big|\Big\langle g_1, U_{\varphi_{\varphi_\zeta(u)}(z)}k_0\Big\rangle\Big|\\
	& \lesssim \|f\|_{\infty}\sup_{\zeta\in \mathbb B_n} \|k_{a(u,u',\zeta,z)}-k_0\|_\infty \lesssim \|f\|_\infty \sup_{\zeta\in \mathbb B_n} \beta\big(a(u,u',\zeta,z),0\big)\\
&\leqslant  \|f\|_{\infty}\tanh^{-1}\Big[\frac{G}{(1-|z|)^2 }\tanh\big(\beta(u,u')\big)\Big],
\end{split}
\end{align}
where the second inequality is due to Lemma \ref{bounded} (note that $|f(\varphi_\zeta(u))|$ is  dominated by $\|f\|_\infty$, so Lemma \ref{bounded} still can be applied here),  and  the third inequality  comes from Lemma \ref{difference}.
Moreover, we have that
\begin{align*}
|a(u,u',\zeta,z)| &= \rho\big(a(u,u',\zeta,z), 0\big)
= \tanh\big[\beta(a(u,u',\zeta,z), 0)\big] \\
&=\frac{G}{(1-|z|)^2 }\tanh\big(\beta(u,u')\big)<\frac{1}{2},
\end{align*}
since $\beta(u,u')< \tanh^{-1}\big(\frac{(1-|z|)^2}{2G}\big)$. This gives $\|k_{a(u,u',\zeta,z)}\|_{\infty}\lesssim 1. $
Then using the same method as used in (\ref{2}), we obtain that
\begin{align*}
	&\bigg\| \sum_{\zeta\in \Gamma_j'}\Big[[f(\varphi_\zeta(u'))-f(\varphi_\zeta(u))] U_{\varphi_{\varphi_\zeta(u)}(z)}k_{a(u,u',\zeta,z)}\otimes U_{\varphi_{\varphi_\zeta(u)}(z)} k_{a(u,u',\zeta,z)}\bigg\|\\
	&\lesssim \sup_{\zeta\in \mathbb B_n}|f(\varphi_\zeta(u))-f(\varphi_\zeta(u'))|
\end{align*}
and
\begin{align*}
	&\bigg\| \sum_{\zeta\in \Gamma_j'} \Big[f(\varphi_\zeta(u))U_{\varphi_{\varphi_\zeta(u)}(z)}k_{a(u,u',\zeta,z)}\otimes U_{\varphi_{\varphi_\zeta(u)}(z)}  (k_{a(u,u',\zeta,z)}-k_0)\Big]\bigg\|\\
	&\lesssim  \|f\|_{\infty}\tanh^{-1}\Big[ \frac{G}{(1-|z|)^2 }\tanh\big(\beta(u,u')\big)\Big].
\end{align*}
Since $f\in \mathrm{BUC}(\mathbb{B}_n)$ and $\beta$ is M\"{o}bius invariant, we deduce that the mapping defined in (\ref{m2}) is uniformly continuous. Similarly, we can show that this mapping is also uniformly bounded.

Finally,  from
\begin{align*}
\int_{\mathbb B_n} f(u) k_{\varphi_{u}(z)}\otimes k_{\varphi_{u}(z)}d\lambda(u)=\sum_{j=1}^m \int_{D(0, 2c)}\Psi_0(u)\sum_{\zeta\in \Gamma_j'}f(\varphi_\zeta(u)) k_{\varphi_{\varphi_\zeta(u)}(z)}\otimes k_{\varphi_{\varphi_\zeta(u)}(z)}d\lambda(u),
\end{align*}
we conclude that the above integral is convergent in the norm topology. Since $\{\varphi_{\zeta}(u): \zeta\in \Gamma_j'\}$ is $c$-separated for any $j\in \{1, 2, \cdots, m\}$, it follows from (\ref{tensoralgebra}) that
$$\int_{\mathbb B_n} f(u) k_{\varphi_{u}(z)}\otimes k_{\varphi_{u}(z)}d\lambda(u)\in \mathcal{T}^b_{lin}(\mathcal{I'}).$$
This completes the proof of Proposition \ref{tensoroperartor}.
\end{proof}

\begin{thm}\label{theorembergman}
Let $\mathcal{I}$ be a translation invariant closed subalgebra of $\mathrm{BUC}(\mathbb{B}_n)$ and $\mathcal{I}'$ be defined in (\ref{generatedspace}). Let $T$  be a bounded linear operator on $L_a^p$. If $T\in \mathcal{A}^{p}_s$ and $\widetilde{T}\in \mathcal{I}$, then
$ T\in \mathcal{T}^b_{lin}(\mathcal{I'}).$
\end{thm}
\begin{proof}
From our assumption  and Corollary \ref{biberezin}, we have that  $\big\langle TU_{(\cdot)}k_w,U_{(\cdot)}k_z \big\rangle\in \mathcal{I}$ for all $z,w\in \mathbb B_n$. By Proposition \ref{tensoroperartor}, we obtain
$$\int_{\mathbb B_n} \langle TU_uk_0,U_u k_{v}\rangle (U_u k_{v})\otimes (U_uk_0) d\lambda(u) \in \mathcal{T}^b_{lin}(\mathcal{I'}).$$
Therefore,  we conclude by Theorem \ref{integralrepresantation} that
$$T=\lim_{r\rightarrow\infty} \int_{D(0,r)}\int_{\mathbb B_n} \langle TU_uk_0,U_u k_{v}\rangle (U_u k_{v})\otimes (U_uk_0) d\lambda(u) d\lambda(v)\in\mathcal{T}^b_{lin}(\mathcal{I'}).$$
This completes the proof of Theorem \ref{theorembergman}.
\end{proof}

We next show that the Berezin transform of the finite product of Toeplitz operators with translation invariant symbols is also translation invariant on $\mathbb B_n$.
\begin{prop}\label{productberezin}
Let $\mathcal{I}$ be a translation invariant closed subalgebra of $\mathrm{BUC}(\mathbb{B}_n)$. For any $\varphi_1, \varphi_2, \cdots, \varphi_m$ in $\mathcal{I}$, let
 $A=T_{\varphi_1}T_{\varphi_2}\cdots T_{\varphi_m}$. Then we have
$$\widetilde{A}\in \mathcal{I}.$$
\end{prop}
\begin{proof}
We prove this proposition by induction.  First, let us consider the case of $m=1$. Recall that
$$\widetilde{T_{\varphi_1}}(z)=\langle T_{\varphi_1}k_z,k_z \rangle=\int_{\mathbb B_n} \varphi_{1}\circ \varphi_z(w)dv(w).$$
Since $\varphi_1\in \mathrm{BUC}(\mathbb{B}_n)$ and
$$\beta(\varphi_z(w),\varphi_z(w'))=\beta(w,w'),$$
we obtain that  the mapping
$$w\mapsto \varphi_{1}\circ \varphi_{(\cdot)}(w)$$
is uniformly continuous on $\mathbb B_n$.  Note that $\varphi_{1}\circ \varphi_{(\cdot)}(w)\in \mathcal{I}$ as $\varphi_1\in \mathcal {I}$. Moreover, since
$$\big\|\varphi_1\circ\varphi_{(\cdot)}(w)\big\|_{\mathcal{I}}= \sup_{z\in \mathbb B_n}|\varphi_1\circ \varphi_{z}(w)|= \|\varphi_1\|_{\infty}< \infty,$$
we have that  the integral
 $$\int_{\mathbb B_n} \varphi_{1}\circ \varphi_{(\cdot)}(w)dv(w)$$ is convergent in the norm topology of $\mathcal{I}$. It follows that
$$\widetilde{T_{\varphi_1}}(\cdot)=\int_{\mathbb B_n} \varphi_{1}\circ \varphi_{(\cdot)}(w)dv(w)\in \mathcal{I}.$$
This shows that our conclusion holds for $m=1$.

Now we suppose that the conclusion holds for $m=k-1$. Let us consider  the case of $m=k$.
\begin{align*}
\langle T_{\varphi_1}T_{\varphi_2}\cdots T_{\varphi_k} k_z,k_z \rangle&=\langle T_{\varphi_2}\cdots T_{\varphi_k} k_z,T_{\overline{\varphi_1}}k_z \rangle\\
&=\int_{\mathbb B_n} \langle T_{\varphi_2}\cdots T_{\varphi_k} k_z,k_w\rangle \langle k_w, T_{\overline{\varphi_1}}k_z\rangle d\lambda(w)\\
&=\int_{\mathbb B_n} \langle T_{\varphi_2}\cdots T_{\varphi_k} k_z,k_{\varphi_z(w)}\rangle \langle k_{\varphi_z(w)}, T_{\overline{\varphi_1}}k_z\rangle d\lambda(w)\\
&=\int_{\mathbb B_n} \langle T_{\varphi_2}\cdots T_{\varphi_k} U_zk_0,U_zk_{w}\rangle \langle T_{\varphi_1}U_zk_{w},U_zk_0\rangle d\lambda(w).
\end{align*}
By our  hypothesis, $\big\langle T_{\varphi_2}\cdots T_{\varphi_k}k_{(\cdot)},k_{(\cdot)}\big\rangle\in \mathcal{I}$ and $\tau_{w}\big\langle T_{\varphi_2}\cdots T_{\varphi_k}k_{(\cdot)},k_{(\cdot)}\big\rangle\in \mathcal{I}$.  It follows that
$$\big\langle T_{\varphi_2}\cdots T_{\varphi_k} U_{(\cdot)} k_w,U_{(\cdot)}k_{w}\big\rangle=\big\langle T_{\varphi_2}\cdots T_{\varphi_k}k_{\varphi_{(\cdot)}(w)},k_{\varphi_{(\cdot)}(w)}\big\rangle\in \mathcal{I}, \ \ \ \  w\in \mathbb B_n,$$
since $\mathcal{I}$ is translation invariant. Using Corollary \ref{biberezin}, we obtain
$$\big\langle T_{\varphi_2}\cdots T_{\varphi_k} U_{(\cdot)} k_0,U_{(\cdot)}k_{w}\big\rangle\in \mathcal{I}.$$
Similarly, we have $$\big\langle T_{\varphi_1}U_{(\cdot)}k_{w},U_{(\cdot)}k_0\big\rangle\in \mathcal{I},$$
to obtain
$$\big\langle T_{\varphi_2}\cdots T_{\varphi_k} U_{(\cdot)} k_0,U_{(\cdot)}k_{w}\big\rangle\big\langle T_{\varphi_1}U_{(\cdot)}k_{w},U_{(\cdot)}k_0\big\rangle\in \mathcal{I}.$$
It follows from  Lemmas \ref{bounded} and \ref{difference}  that the mapping
$$w \mapsto \big\langle T_{\varphi_2}\cdots T_{\varphi_k} U_{(\cdot)} k_0,U_{(\cdot)}k_{w}\big\rangle\big\langle T_{\varphi_1}U_{(\cdot)}k_{w},U_{(\cdot)}k_0\big\rangle$$
is uniformly continuous from each  $D(0, r)$ to $\mathcal{I}$.

Choosing $t>2>t'>1$ such that $\frac{1}{t}+\frac{1}{t'}=1$, then we  have by (\ref{U}) that
\begin{align*}
&\Big\|\langle T_{\varphi_2}\cdots T_{\varphi_k} U_{(\cdot)} k_0,U_{(\cdot)}k_{w}\rangle\langle T_{\varphi_1}U_{(\cdot)}k_{w},U_{(\cdot)}k_0\rangle\Big\|\\
&=\sup_{z\in \mathbb B_n}\Big|\langle T_{\varphi_2}\cdots T_{\varphi_k} U_{z} k_0,U_{z}k_{w}\rangle\langle T_{\varphi_1}U_{z}k_{w},U_{z}k_0\rangle\Big|\\
&=\sup_{z\in \mathbb B_n}\Big|\langle T_{\varphi_2\circ\varphi_z}\cdots T_{\varphi_k\circ\varphi_z} k_0, k_{w}\rangle
\langle k_{w},T_{\overline{\varphi_1}\circ\varphi_z}k_0\rangle\Big|\\
& \leqslant  \|T_{\varphi_2\circ\varphi_z}\cdots T_{\varphi_k\circ\varphi_z}1\|_{t'}\|k_w\|_t \|k_w\|_t\|T_{\overline{\varphi_1}\circ\varphi_z}1\|_{t'}\\
& \lesssim (1-|w|^2)^{(1-\frac{2}{t})(n+1)},
\end{align*}
where the  function $(1-|w|^2)^{(1-\frac{2}{t})(n+1)}$ is integrable with respect to the measure $d\lambda$ on $\mathbb B_n$. Therefore,
$$\big\langle T_{\varphi_1}T_{\varphi_2}\cdots T_{\varphi_k} k_{(\cdot)},k_{(\cdot)} \big\rangle=\int_{\mathbb B_n} \big\langle T_{\varphi_2}\cdots T_{\varphi_k} U_{(\cdot)}k_0,U_{(\cdot)}k_{w}\big\rangle \big\langle T_{\varphi_1}U_{(\cdot)}k_{w},U_{(\cdot)}k_0\big\rangle d\lambda(w)$$
is also in  $\mathcal{I}$, as desired. This completes the proof of  Proposition \ref{productberezin}.
\end{proof}

Given a translation invariant closed subalgebra  $\mathcal{I}\subset \mathrm{BUC}(\mathbb{B}_n)$, we recall that  $\mathcal{I'}$ is defined by (\ref{generatedspace}).  Then combining Theorem \ref{theorembergman} with Proposition \ref{productberezin}  gives that the Toeplitz algebra $\mathcal{T}^b(\mathcal{I})$ is contained in the closed linear space $\mathcal{T}^b_{lin}(\mathcal{I'})$.

\begin{prop}\label{theorem2}
Suppose that $\mathcal{I}$ is a translation invariant closed subalgebra of $\mathrm{BUC}(\mathbb{B}_n)$. Then
$$\mathcal{T}^b(\mathcal{I})\subset \mathcal{T}^b_{lin}(\mathcal{I'})$$
holds on the $p$-Bergman space $L_a^p$.
\end{prop}
\begin{proof}
Let $A=T_{\varphi_1}T_{\varphi_2}\cdots T_{\varphi_m}\in \mathcal{T}^b(\mathcal{I})$ with $\varphi_1, \varphi_2, \cdots, \varphi_m\in\mathcal{I}$. It follows from  Proposition \ref{productberezin} that $\widetilde{A}\in \mathcal{I}$. By Propositions 2.2 and 2.3 in \cite{wick2014}, we deduce that $A\in \mathcal{A}^{p}_s.$  Now we conclude by  Theorem \ref{theorembergman} that $A\in \mathcal{T}^b_{lin}(\mathcal{I'})$.  This finishes the proof of Proposition \ref{theorem2}.
\end{proof}

The next theorem generalizes the corresponding result  obtained by Xia \cite[Theorem 1.5]{Xia2015} for the case of $p=2$.
\begin{thm}\label{pbergthm}
On the $p$-Bergman space $L_a^p$, we have that
$$ \mathcal{T}^b_{lin}[\mathrm{BUC}(\mathbb{B}_n)]=\mathcal{T}^b[\mathrm{BUC}(\mathbb{B}_n)]=\mathrm{clos} (\mathcal{A}^{p}_s)$$ and
$$\mathcal{T}^b_{lin}[C_0(\mathbb{B}_n)]=\mathcal{T}^b[C_0(\mathbb{B}_n)]=\mathcal{K},$$
where $\mathrm{clos}(\mathcal{A}^{p}_s)$ denotes the norm closure of $\mathcal{A}^{p}_s$ and $\mathcal{K}$ denotes the set of compact operators on $L^p_a$.
\end{thm}

Before giving the proof of Theorem \ref{pbergthm}, we  mention here that the Berezin transform of any bounded linear operator on $L_a^p$ is in $\text{BUC}(\mathbb B_n)$, see \cite[Proposition 4.9]{Mit} and its proof, or Proposition 8.3 and Lemma 9.1 in \cite{suarez2005} if needed.

\begin{proof}[Proof of Theorem \ref{pbergthm}] First, we note that
$$\mathcal{T}^b_{lin}[\mathrm{BUC}(\mathbb{B}_n)]\subset\mathcal{T}^b[\mathrm{BUC}(\mathbb{B}_n)] \ \ \text{ and } \ \ \mathcal{T}^b_{lin}[C_0(\mathbb{B}_n)]\subset\mathcal{T}^b[C_0(\mathbb{B}_n)].$$
On the other hand,  it follows from   Lemma \ref{lemmaI} and Proposition  \ref{theorem2} that
$$\mathcal{T}^b[\mathrm{BUC}(\mathbb{B}_n)]\subset\mathcal{T}^b_{lin}[\mathrm{BUC}(\mathbb{B}_n)']
\subset\mathcal{T}^b_{lin}[\mathrm{BUC}(\mathbb{B}_n)]$$
and
$$\mathcal{T}^b[C_0(\mathbb{B}_n)]\subset\mathcal{T}^b_{lin}[C_0(\mathbb{B}_n)']\subset\mathcal{T}^b_{lin}[C_0(\mathbb{B}_n)].$$
Thus we obtain
$$\mathcal{T}^b_{lin}[\mathrm{BUC}(\mathbb{B}_n)]=\mathcal{T}^b[\mathrm{BUC}(\mathbb{B}_n)] \ \ \text{ and } \ \ \mathcal{T}^b_{lin}[C_0(\mathbb{B}_n)]=\mathcal{T}^b[C_0(\mathbb{B}_n)].$$
Moreover, it follows from  \cite[Propositions 2.2 and 2.3]{wick2014} that $$\mathcal{T}^b[\mathrm{BUC}(\mathbb{B}_n)]\subset\mathrm{clos}(\mathcal{A}^{p}_s).$$ Recall that the Berezin transform  $\widetilde{T}\in\mathrm{BUC}(\mathbb{B}_n)$ if $T\in\mathcal{A}^{p}_s$ (see the paragraph before the proof of Theorem \ref{pbergthm}). Therefore,   we deduce by Theorem \ref{theorembergman} that
$$\mathcal{T}^b[\mathrm{BUC}(\mathbb{B}_n)]=\mathrm{clos}(\mathcal{A}^{p}_s) .$$

In order to  finish the proof of this theorem, it remains to show that $$\mathcal{K}=\mathcal{T}^b_{lin}[C_0(\mathbb{B}_n)].$$ Let
$$\mathfrak{F}=\mathrm{span} \big\{k_x \otimes k_y:  x,y\in\mathbb{B}_n\big\}.$$
Since the linear span of the normalized reproducing kernels is dense in $L_a^p$ and $L_a^q$, we have that $\mathfrak{F}$ is dense in $\mathcal K$.
Now we are going to show $\mathfrak{F}\subset \mathcal{A}^{p}_s.$ To do this, we need only to show that $k_x \otimes k_y$ is $s$-weakly localized for all $x, y\in \mathbb B_n$.

In fact, for every $k_x \otimes k_y$ we have
$$|\langle (k_x \otimes k_y)k_z,k_w\rangle|=|\langle k_z,k_y\rangle|~|\langle k_x,k_w\rangle|.$$
Suppose that $x,y\in D(0,a)$ for some $a>0$, then we have by Lemma \ref{discrete} that
$$|\langle k_z,k_y\rangle|~|\langle k_x,k_w\rangle|\leqslant C_a |\langle k_z,k_\xi\rangle|~|\langle k_\xi,k_w\rangle|$$
 for all $\xi\in D(0,a)$,
where $C_a$ is a positive  constant depending only on $a$. This gives that
\begin{align*}
|\langle k_z,k_y\rangle|~|\langle k_x,k_w\rangle|
&\leqslant \frac{C_a}{\lambda(D(0,a))}\int_{D(0,a)} |\langle k_z,k_\xi\rangle|~|\langle k_\xi,k_w\rangle|d\lambda(\xi)\\
&\leqslant \frac{C_a}{\lambda(D(0,a))}\int_{\mathbb{B}_n} |\langle k_z,k_\xi\rangle|~|\langle k_\xi,k_w\rangle|d\lambda(\xi),
\end{align*}
where $\lambda(D(0, a))=\int_{D(0, a)} d\lambda$. Thus we have
\begin{align*}
&\sup_{z\in \mathbb{B}_n}\int_{\mathbb{B}_n\setminus D(z,r)}|\langle (k_x \otimes k_y)k_z,k_w\rangle|\frac{\|K_z\|_2^{1-\frac{2s}{q(n+1)}}}{\|K_w\|_2^{1-\frac{2s}{q(n+1)}}}d\lambda(w)\\
& \lesssim \sup_{z\in \mathbb{B}_n}\int_{\mathbb{B}_n\setminus D(z,r)}\int_{\mathbb{B}_n} |\langle k_z,k_\xi\rangle|~|\langle k_\xi,k_w\rangle|d\lambda(\xi)\frac{\|K_z\|_2^{1-\frac{2s}{q(n+1)}}}{\|K_w\|_2^{1-\frac{2s}{q(n+1)}}}d\lambda(w)\\
&= \sup_{z\in \mathbb{B}_n}\int_{\mathbb{B}_n\setminus D(z,r)}\int_{\mathbb{B}_n} |\langle k_z,k_\xi\rangle|~|\langle k_\xi,k_w\rangle|\frac{\|K_z\|_2^{1-\frac{2s}{q(n+1)}}}{\|K_\xi\|_2^{1-\frac{2s}{q(n+1)}}}
\frac{\|K_\xi\|_2^{1-\frac{2s}{q(n+1)}}}{\|K_w\|_2^{1-\frac{2s}{q(n+1)}}}d\lambda(\xi)d\lambda(w)\\
&= \sup_{z\in \mathbb{B}_n}\int_{\mathbb{B}_n\setminus D(z,r)}\int_{D(z,\frac{r}{2})} |\langle k_z,k_\xi\rangle|~|\langle k_\xi,k_w\rangle|\frac{\|K_z\|_2^{1-\frac{2s}{q(n+1)}}}{\|K_\xi\|_2^{1-\frac{2s}{q(n+1)}}}
\frac{\|K_\xi\|_2^{1-\frac{2s}{q(n+1)}}}{\|K_w\|_2^{1-\frac{2s}{q(n+1)}}}d\lambda(\xi)d\lambda(w)\\
& \ +\sup_{z\in \mathbb{B}_n}\int_{\mathbb{B}_n\setminus D(z,r)}\int_{\mathbb{B}_n\setminus D(z,\frac{r}{2})} |\langle k_z,k_\xi\rangle|~|\langle k_\xi,k_w\rangle|\frac{\|K_z\|_2^{1-\frac{2s}{q(n+1)}}}{\|K_\xi\|_2^{1-\frac{2s}{q(n+1)}}}
\frac{\|K_\xi\|_2^{1-\frac{2s}{q(n+1)}}}{\|K_w\|_2^{1-\frac{2s}{q(n+1)}}}d\lambda(\xi)d\lambda(w)\\
&:=I_1(r)+I_2(r).
\end{align*}

For $I_1(r)$, Fubini's theorem gives us
\begin{align*}
I_1(r)&=\sup_{z\in \mathbb{B}_n}\int_{\mathbb{B}_n\setminus D(z,r)}\int_{D(z, \frac{r}{2})} |\langle k_z,k_\xi\rangle|~|\langle k_\xi,k_w\rangle|\frac{\|K_z\|_2^{1-\frac{2s}{q(n+1)}}}{\|K_\xi\|_2^{1-\frac{2s}{q(n+1)}}}
\frac{\|K_\xi\|_2^{1-\frac{2s}{q(n+1)}}}{\|K_w\|_2^{1-\frac{2s}{q(n+1)}}}d\lambda(\xi)d\lambda(w)\\
&=\sup_{z\in \mathbb{B}_n}\int_{D(z,\frac{r}{2})}\int_{\mathbb{B}_n\setminus D(z,r)} |\langle k_z,k_\xi\rangle|~|\langle k_\xi,k_w\rangle|\frac{\|K_z\|_2^{1-\frac{2s}{q(n+1)}}}{\|K_\xi\|_2^{1-\frac{2s}{q(n+1)}}}
\frac{\|K_\xi\|_2^{1-\frac{2s}{q(n+1)}}}{\|K_w\|_2^{1-\frac{2s}{q(n+1)}}}d\lambda(w)d\lambda(\xi)\\
&\leqslant \sup_{z\in \mathbb{B}_n}\int_{D(z,\frac{r}{2})}\int_{\mathbb{B}_n\setminus D(\xi, \frac{r}{2})} |\langle k_z,k_\xi\rangle|~
|\langle k_\xi,k_w\rangle|\frac{\|K_z\|_2^{1-\frac{2s}{q(n+1)}}}{\|K_\xi\|_2^{1-\frac{2s}{q(n+1)}}}
\frac{\|K_\xi\|_2^{1-\frac{2s}{q(n+1)}}}{\|K_w\|_2^{1-\frac{2s}{q(n+1)}}}d\lambda(w)d\lambda(\xi)\\
&\leqslant\bigg[\sup_{\xi\in \mathbb{B}_n}\int_{\mathbb{B}_n\setminus D(\xi,\frac{r}{2})}|\langle k_\xi,k_w\rangle|\frac{\|K_\xi\|_2^{1-\frac{2s}{q(n+1)}}}{\|K_w\|_2^{1-\frac{2s}{q(n+1)}}}d\lambda(w)\bigg]\bigg[\sup_{z\in \mathbb{B}_n}
\int_{\mathbb B_n}|\langle k_z,k_\xi\rangle|\frac{\|K_z\|_2^{1-\frac{2s}{q(n+1)}}}{\|K_\xi\|_2^{1-\frac{2s}{q(n+1)}}}d\lambda(\xi)\bigg],
\end{align*}
where the first inequality follows from that $D(\xi, \frac{r}{2})\subset D(z, r)$ for $\xi\in D(z, \frac{r}{2})$.

To estimate $I_2(r)$, we note that
\begin{align*}
I_2(r)&=\sup_{z\in \mathbb{B}_n}\int_{\mathbb{B}_n\setminus D(z,r)}\int_{\mathbb{B}_n\setminus D(z,\frac{r}{2})} |\langle k_z,k_\xi\rangle|~|\langle k_\xi,k_w\rangle|\frac{\|K_z\|_2^{1-\frac{2s}{q(n+1)}}}{\|K_\xi\|_2^{1-\frac{2s}{q(n+1)}}}
\frac{\|K_\xi\|_2^{1-\frac{2s}{q(n+1)}}}{\|K_w\|_2^{1-\frac{2s}{q(n+1)}}}d\lambda(\xi)d\lambda(w)\\
&\leqslant \sup_{z\in \mathbb{B}_n}\int_{\mathbb{B}_n}\int_{\mathbb{B}_n\setminus D(z, \frac{r}{2})} |\langle k_z,k_\xi\rangle|~
|\langle k_\xi,k_w\rangle|\frac{\|K_z\|_2^{1-\frac{2s}{q(n+1)}}}{\|K_\xi\|_2^{1-\frac{2s}{q(n+1)}}}
\frac{\|K_\xi\|_2^{1-\frac{2s}{q(n+1)}}}{\|K_w\|_2^{1-\frac{2s}{q(n+1)}}}d\lambda(\xi)d\lambda(w)\\
&\leqslant \bigg[\sup_{z\in \mathbb{B}_n}\int_{\mathbb{B}_n\setminus D(z,\frac{r}{2})} |\langle k_z,k_\xi\rangle|\frac{\|K_z\|_2^{1-\frac{2s}{q(n+1)}}}{\|K_\xi\|_2^{1-\frac{2s}{q(n+1)}}}d\lambda(\xi)\bigg]\bigg[
\sup_{\xi\in \mathbb{B}_n}\int_{\mathbb{B}_n}|\langle k_\xi,k_w\rangle|\frac{\|K_\xi\|_2^{1-\frac{2s}{q(n+1)}}}{\|K_w\|_2^{1-\frac{2s}{q(n+1)}}}d\lambda(w)\bigg].
\end{align*}
 By  \cite[Lemma 2.1]{wick2014}, we get   that
\begin{align*}
&\lim_{r\rightarrow\infty}\sup_{z\in \mathbb{B}_n}\int_{\mathbb{B}_n\setminus D(z,r)}|\langle (k_x \otimes k_y)k_z,k_w\rangle|\frac{\|K_z\|_2^{1-\frac{2s}{q(n+1)}}}{\|K_w\|_2^{1-\frac{2s}{q(n+1)}}}d\lambda(w)\\
&\lesssim \lim_{r\rightarrow\infty}I_1(r)+\lim_{r\rightarrow\infty}I_2(r)=0.
\end{align*}
Using the same arguments as above, one can verify that $k_x\otimes k_y$ satisfies the other three conditions in Definition \ref{s-weakly}. Thus  we obtain that
$$k_x \otimes k_y\in \mathcal{A}^{p}_s,$$
and hence $\mathfrak{F}\subset\mathcal{A}^{p}_s$.

For any $T\in\mathfrak{F}\subset \mathcal{K}$, we have that $\widetilde{T}\in C_{0}(\mathbb{B}_n)$. Since $T\in\mathfrak{F}\subset \mathcal{A}^{p}_s$, Lemma \ref{lemmaI} and Theorem \ref{theorembergman} imply that
$$T\in\mathcal{T}^b_{lin}[C_0(\mathbb{B}_n)],$$
to get $\mathfrak{F}\subset\mathcal{T}^b_{lin}[C_0(\mathbb{B}_n)]$. This gives  that $$\mathcal{K}=\mathrm{clos}(\mathfrak{F})\subset\mathcal{T}^b_{lin}[C_0(\mathbb{B}_n)].$$
Since $\mathcal{T}^b_{lin}[C_0(\mathbb{B}_n)]\subset\mathcal{K}$, we finally  obtain that
$$\mathcal{K}=\mathcal{T}^b_{lin}[C_0(\mathbb{B}_n)].$$
This completes the proof of Theorem \ref{pbergthm}.
\end{proof}
Let us  make a remark for  the above theorem.\vspace{2mm}\\
\textbf{Remark.} As mentioned in Section \ref{Intro}, it was proved by Xia \cite[Theorem 1.5]{Xia2015} that
\begin{equation}\label{2bergmanthm}
\mathcal{T}^b[L^{\infty}(\mathbb B_n, dv)]=\mathrm{clos}(\mathcal{A}^{2}_s)=\mathcal{T}^b_{lin}[L^{\infty}(\mathbb B_n, dv)]
\end{equation}
 and  $\mathrm{clos}(\mathcal{A}^{2}_s)$ equals the $C^*$-algebra generated by  $\mathcal{A}^{2}_s$ for the case of  $L_a^2$. But Su\'{a}rez obtained in \cite[Theorem 7.3]{suarez} that
$$\mathcal{T}^b_{lin}[\mathrm{BUC}(\mathbb{B}_n)]=\mathcal{T}^b_{lin}[L^{\infty}(\mathbb B_n, dv)] \ \ \ \ \text{ and }\ \ \  \ \mathcal{T}^b[\mathrm{BUC}(\mathbb{B}_n)]=\mathcal{T}^b[L^{\infty}(\mathbb B_n, dv)]$$
hold on the $p$-Bergman space $L_a^p$. Therefore,  (\ref{2bergmanthm}) can be rewritten as:
$$\mathcal{T}^b[\mathrm{BUC}(\mathbb{B}_n)]=\mathrm{clos}(\mathcal{A}^{2}_s)=\mathcal{T}^b_{lin}[\mathrm{BUC}(\mathbb{B}_n)] \ \  \mathrm{on\ the \ Bergman\ space} \  L_a^2.$$
This means that Theorem \ref{pbergthm} extends  Xia's result to case of the $p$-Bergman space with $1<p<\infty$.\\

Let $1<p<\infty$ and $\mathrm{VMO}^p(\mathbb{B}_n)$ denote the collection  of vanishing mean oscillation functions on the unit ball (see Section 5 of \cite{Pa}). Denote
$$\mathrm{VMO}^p_{\infty}(\mathbb{B}_n)=\mathrm{VMO}^p(\mathbb{B}_n)\cap L^{\infty}(\mathbb B_n, dv).$$
Furthermore, \cite[Theorem 6.1]{Pa} tells us that the Hankel operator $H_{\varphi}$ is compact on $L_a^p$ ($1<p<\infty$) whenever $\varphi\in \mathrm{VMO}^p_{\infty}(\mathbb{B}_n)$. It follows that
$$T_{\varphi_1}T_{\varphi_2}-T_{\varphi_1 \varphi_2}=-H_{\overline{\varphi_1}}^*H_{\varphi_2}$$ is compact on $L_a^p$ if $\varphi_1$ and $ \varphi_2$ are both in $\mathrm{VMO}^p_{\infty}(\mathbb{B}_n).$

 We end this section by the following corollary, which gives a
characterization for the Toeplitz algebra with symbols in an algebra lying between $C_0(\mathbb{B}_n)$ and  $\mathrm{VMO}^p_{\infty}(\mathbb{B}_n)$.
\begin{cor}\label{vmoalgebra}
Let $1<p<\infty$  and $\mathcal{I}$ be an algebra such that
$$C_0(\mathbb{B}_n)\subset \mathcal{I}\subset \mathrm{VMO}^p_{\infty}(\mathbb{B}_n).$$
Then we have
$$\mathcal{T}^b(\mathcal{I})=\mathrm{clos}\{T_\varphi+K:  \varphi\in\mathcal{I}  \ \text{ and } \  K\in \mathcal{K} \}=\mathcal{T}^b_{lin}(\mathcal{I}),$$
where $\mathrm{clos}\{T_\varphi+K:  \varphi\in\mathcal{I}\ \text{ and } \ K\in \mathcal{K} \}$ denotes the norm  closure of $\{T_\varphi+K:  \varphi\in\mathcal{I}\ \text{ and } \ K\in \mathcal{K} \}$.
\end{cor}
\begin{proof}
First, we claim that
$$\mathcal{T}^b(\mathcal{I})\subset \mathrm{clos}\{T_\varphi+K:  \varphi \in\mathcal{I}\  \text{ and }\  K\in \mathcal{K} \}.$$
Let us assume this for the moment and we will give its proof later. As $\mathcal{T}_{lin}^b(\mathcal{I})\subset \mathcal{T}^b(\mathcal{I})$, it is enough to show that
$$\mathrm{clos}\{T_\varphi+K:  \varphi\in\mathcal{I} \ \text{ and } \  K\in \mathcal{K} \}\subset \mathcal{T}_{lin}^b(\mathcal{I}).$$
According to  Theorem \ref{pbergthm}, we have
$$\mathcal{K}=\mathcal{T}^b_{lin}[C_0(\mathbb{B}_n)],$$
to obtain
\begin{align*}
\mathrm{clos}\{T_\varphi+K:  \varphi\in\mathcal{I} \ \text{ and } \  K\in \mathcal{K} \}&= \mathrm{clos}\{T_\varphi+T_\psi:  \varphi\in\mathcal{I} \ \text{ and }\  \psi\in C_0(\mathbb{B}_n) \}\\
&\subset \mathrm{clos}\{T_\varphi: \varphi \in\mathcal{I} \} \ \  \  \quad\text{\big(since \ $ C_0(\mathbb{B}_n)\subset\mathcal{I}$\big)}\\
&=\mathcal{T}_{lin}^b(\mathcal{I}).
\end{align*}

In order to complete the proof, we shall prove the above claim by induction. First, we have
$$T_{\varphi_1} \in \mathrm{clos}\{T_\varphi +K:  \varphi\in\mathcal{I}  \ \text{ and } \  K\in \mathcal{K} \}$$
if  $\varphi_1\in \mathcal{I}$. Then we suppose that
$$T_{\varphi_1}\cdots T_{\varphi_{m-1}}\in \mathrm{clos}\{T_\varphi+K:  \varphi\in\mathcal{I}\  \text{ and }  \ K\in \mathcal{K} \}$$
 for any  functions $\varphi_1,\cdots, \varphi_{m-1}\in \mathcal{I}$.
Let  $\varphi_m$ be in $\mathcal{I}$, then  we have
$$T_{\varphi_1}\cdots T_{\varphi_{m-1}}T_{\varphi_m}=T_{\varphi_1}\cdots T_{\varphi_{m-2}} T_{\varphi_{m-1}\varphi_m}+T_{\varphi_1}\cdots T_{\varphi_{m-2}}(T_{\varphi_{m-1}}T_{\varphi_m}-T_{\varphi_{m-1}\varphi_m}).$$
Note that $T_{\varphi_{m-1}}T_{\varphi_m}-T_{\varphi_{m-1}\varphi_m}$ is compact on $L_a^p$, since  $\varphi_{m-1} $ and $\varphi_m$ are in  $\mathcal{I}\subset \mathrm{VMO}^p_{\infty}(\mathbb{B}_n)$. By the induction hypothesis, we have
$$T_{\varphi_1}\cdots T_{\varphi_{m-1}\varphi_m}\in \mathrm{clos}\{T_\varphi+K:  \varphi\in\mathcal{I} \ \text{ and } \ K\in \mathcal{K} \}.$$
This yields  that
$$T_{\varphi_1}\cdots T_{\varphi_{m-1}}T_{\varphi_m}\in \mathrm{clos}\{T_\varphi+K: \varphi\in\mathcal{I} \ \text{ and } \ K\in \mathcal{K} \},$$
to get
$$\mathcal{T}^b(\mathcal{I})\subset \mathrm{clos}\{T_\varphi+K:  \varphi \in\mathcal{I} \ \text{ and } \  K\in \mathcal{K} \}.$$
This proves the claim and so the proof of Corollary \ref{vmoalgebra} is finished.
\end{proof}
\vspace{3mm}

\subsection*{Acknowledgments} We would like to thank the reviewer for the constructive comments and suggestions that improved the
content of this paper.   We are  also grateful to Professor Jingbo Xia (SUNY-Buffalo) and Professor Kunyu Guo (Fudan University) for providing useful suggestions.  This work was partially supported by  NSFC (grant number: 11701052). The second author was partially supported by the Fundamental Research Funds for the Central Universities (grant numbers: 2020CDJQY-A039, 2020CDJ-LHSS-003).


\begin{thebibliography}{99}

\bibitem{Bauer} W. Bauer and  J. Isralowitz, \emph{Compactness characterization of operators in the Toeplitz algebra of the Fock space $F^p_{\alpha}$},  {\sl J. Funct. Anal.} {\bf 263} (2012), 1323-1355.


\bibitem{Englis} M. Engli\v{s}, \emph{Density of algebras generated by Toeplitz operator on Bergman space},  {\sl Ark. Mat.} {\bf 30} (1992), 227-243.

\bibitem{Robert} R. Fulsche, \emph{Correspondence theory on $p$-Fock spaces with applications to Toeplitz algebras},
{\sl J. Funct. Anal.} {\bf 279} (2020), Article 108661.

\bibitem{Hagger}R. Hagger, \emph{Essential commutants and characterizations of the Toeplitz algebra}, arXiv: 2002.02344.

\bibitem{wick2014} J. Isralowitz, M. Mitkovski  and  B. D. Wick, \emph{Localization and compactness in Bergman and Fock spaces},  {\sl Indiana Univ. Math. J.} {\bf 64} (2015), 1553-1573.

\bibitem{Mit} M. Mitkovski, D. Su\'{a}rez and B. D. Wick, \emph{The essential norm of operators on ${A^p_\alpha (\mathbb {B}_n)}$}, {\sl Integral Equ. Oper. Theory} {\bf75}  (2013),  197-233.




 \bibitem{Pa} J. Pau, R. Zhao and  K. Zhu, \emph{Weighted BMO and Hankel operators between Bergman spaces},  {\sl Indiana Univ. Math. J.}  {\bf 65} (2016), 1639-1673.



\bibitem{rudin} W. Rudin, \emph{Function Theory in the Unit Ball of $\mathbb{C}^n$}, Springer-Verlag, Berlin, 2008.

\bibitem{suarez2005} D. Su\'{a}rez, \emph{Approximation and $n$-Berezin transform of operators on the Bergman space},  {\sl J. Reine Angew. Math.} {\bf 581} (2005), 175-192.

\bibitem{suarez} D. Su\'{a}rez, \emph{The essential norm of operators in the Toeplitz algebra on $A^p(\mathbb{B}_n)$},  {\sl Indiana Univ. Math. J.} {\bf 56} (2007), 2185-2232.

\bibitem{XiaZheng} J. Xia and  D. Zheng, \emph{Localization and Berezin transform on the Fock space},  {\sl J. Funct. Anal.} {\bf264} (2013), 97-117.

\bibitem{Xia2015} J. Xia, \emph{Localization and the Toeplitz algebra on the Bergman space},  {\sl J. Funct. Anal.} {\bf 269} (2015), 781-814.

\bibitem{Xia2017} J. Xia, \emph{On the essential commutant of the Toeplitz algebra on the Bergman space},  {\sl J. Funct. Anal.} {\bf 272} (2017), 5191-5217.

\bibitem{Xia2018} J. Xia, \emph{A double commutant relation in the Calkin algebra on the Bergman space}, {\sl J. Funct. Anal.} {\bf274} (2018), 1631-1656.

\bibitem{Zhu} K. Zhu, \emph{Analysis on Fock Spaces}, Graduate Texts in Mathematics, vol. {\bf 263}, Springer, New York, 2012.

\bibitem{Zhu2} K. Zhu, \emph{Spaces of Holomorphic Functions in the Unit Ball},  Graduate Texts in Mathematics, vol. {\bf 226}, Springer-Verlag, New
York, 2005.

\end{thebibliography}
\end{document}